\documentclass[reqno,12pt]{amsart}

\usepackage{tikz,young}
\usepackage{amsmath}
\usepackage{amsfonts}
\usepackage{amssymb,euscript, mathrsfs,amscd} 
\usepackage{epsfig}
\usepackage{a4wide}
\usepackage{relsize}
\usepackage{blkarray}
\usepackage{multirow}
\usepackage{mathtools}
\usepackage{bm}
\usepackage{ytableau}
\usepackage{caption}
\usepackage{hyperref}
\usepackage{cases}
\usepackage{ragged2e}
\usepackage{enumitem}
\usepackage[normalem]{ulem}
\usepackage{subfigure}
\usepackage{graphicx} 
\usepackage[colorinlistoftodos, textwidth = 2.3cm]{todonotes}

\usetikzlibrary{decorations.markings}

\newtheorem{thm}{Theorem}[section]
\newtheorem{prop}[thm]{Proposition}
\newtheorem{lem}[thm]{Lemma}
\newtheorem{cor}[thm]{Corollary}
\newtheorem{conj}[thm]{Conjecture}
\theoremstyle{definition}
\newtheorem{defn}[thm]{Definition}

\newtheorem{rmk}[thm]{Remark}
\theoremstyle{remark}

\newcommand\CC{\mathbb{C}}

\newcommand\Spec{\mathcal{S}}

\newcommand\SYT{\operatorname{SYT}}
\newcommand\Dyck{\operatorname{Dyck}}
\newcommand\NC{\operatorname{NC}}
\newcommand\NN{\operatorname{NN}}
\newcommand\Mat{\operatorname{Mat}}
\newcommand\Cr{\operatorname{Cr}}
\newcommand\SL{\operatorname{SL}}
\newcommand\AC{\operatorname{AC}}
\newcommand\rlmin{\operatorname{rlmin}}
\newcommand\cc{\operatorname{c}}
\newcommand\Web{\operatorname{Web}}

\newcommand\NS{\mathsf{N}}
\newcommand\ES{\mathsf{E}}

\newcommand\OnlyXmarking[2]
{\draw[very thick] (#1-0.7,#2-0.7)--(#1-0.3,#2-0.3);
\draw[very thick] (#1-0.7,#2-0.3)--(#1-0.3,#2-0.7); }
\newcommand\Xmarking[2]
{\draw[very thick] (#1-0.7,#2-0.7)--(#1-0.3,#2-0.3);
\draw[very thick] (#1-0.7,#2-0.3)--(#1-0.3,#2-0.7);
\draw (#1-1,#2-0.5)--(#1-0.5,#2-0.5);
\draw (#1-0.5,#2-0.5)--(#1-0.5,#2); }
\newcommand\Cross[2]
{\draw (#1-0.5,#2)--(#1-0.5,#2-1);
\draw (#1,#2-0.5)--(#1-1,#2-0.5); }
\newcommand\UP[2]{\draw (#1-0.5,#2)--(#1-0.5,#2-1);}
\newcommand\EAST[2]{\draw (#1,#2-0.5)--(#1-1,#2-0.5);}
\newcommand\Asmooth[2]
{\draw (#1,#2-0.5) .. controls (#1-0.45,#2-0.45) and (#1-0.45,#2-0.45) .. (#1-0.5,#2);
\draw (#1-1,#2-0.5) .. controls (#1-0.55,#2-0.55) and (#1-0.55,#2-0.55) .. (#1-0.5,#2-1); }

\newcommand\Matching[2]{\draw (#1,0) to [out=55,in=125] (#2,0);}

\newcommand\SYM{\mathfrak{S}}

\newcommand\qand{\quad\mbox{and}\quad}

\title[A combinatorial model for the transition matrix]{A combinatorial model for the transition matrix between the Specht and web bases}

\author{Byung-Hak Hwang}
\address{Applied Algebra and Optimization Research Center, Sungkyunkwan University, Suwon,
South Korea}
\email{byunghakhwang@gmail.com}

\author{Jihyeug Jang}
\address{Department of Mathematics,
Sungkyunkwan University (SKKU), Suwon, Gyeonggi-do 16419, South Korea}
\email{4242ab@gmail.com}

\author{Jaeseong Oh}
\address{Korea Institute for Advanced Study,
85 Hoegiro, Dongdaemun-gu, Seoul 02455, South Korea}
\email{jsoh@kias.re.kr}

\begin{document}

\begin{abstract}
  We introduce a new class of permutations, called web permutations.
  Using these permutations, we provide a combinatorial interpretation for 
  entries of the transition matrix between the Specht and web bases,
  which answers Rhoades's question.
  Furthermore, we study enumerative properties of these permutations.
\end{abstract}

\maketitle

\section{Introduction and the main result}
In this article, we study the transition matrix between two famous bases,
the Specht basis and the web basis, for the irreducible representation of
the symmetric group \( \SYM_{2n} \) indexed by the partition \( (n,n) \).
Motivated by Rhoades's work~\cite{Rho19}, we give a combinatorial interpretation
for entries of the transition matrix as a certain class of permutations,
and present their interesting properties.

For an integer \( n\ge 1 \), let \( \SYM_{2n} \) be the symmetric group on the set
\( [2n] = \{ 1,\dots,2n \} \).
It is well known that each irreducible representation of \( \SYM_{2n} \) can be
indexed by a partition of \( 2n \).
For a partition \( \lambda \) of \( 2n \), we then denote
by \( \Spec^\lambda \) the irreducible representation indexed by \( \lambda \), called the \emph{Specht module}.
In this article, we narrow our focus down to the Specht module indexed by the partition \((n,n)\), and two well-studied bases for \( \Spec^{(n,n)} \).

A \emph{standard Young tableau} of shape \( (n,n) \) is an \( 2\times n \) array
of integers whose entries are \( [2n] \), and
each row and each column are increasing. See Figure~\ref{fig:SYT} for example.
\begin{figure}
  \begin{ytableau}
    1 & 3 & 4 & 6 \\
    2 & 5 & 7 & 8
  \end{ytableau} 
  \caption{A standard Young tableau of shape \( (4,4) \).} \label{fig:SYT}
\end{figure}
The set of standard Young tableaux of shape \( (n,n) \), denoted by \( \SYT(n,n) \), parametrizes
the \emph{Specht basis}
\[
  \{ v_T \in \Spec^{(n,n)} : T\in  \SYT(n,n) \}
\]
for \( \Spec^{(n,n)} \). For more details on the Specht basis and
related combinatorics, see \cite{Ful97, Sag01}.

A \emph{(perfect) matching} on \( [2n] \) is a set partition of \( [2n] \) such that
each block has size 2.
We also depict a matching on \( [2n] \) as a diagram consisting of \( 2n \)
vertices and \( n \) arcs where any pair of arcs has no common vertex.
A \emph{crossing} is a pair of arcs \( \{a,c\} \) and \( \{b,d\} \)
with \( a<b<c<d \).
A matching is called \emph{noncrossing} if the matching has no crossing,
and \emph{nonnesting} if there is no pair of arcs \( \{a,d\} \)
and \( \{b,c\} \) with \( a<b<c<d \); see Figure~\ref{fig:matching}.
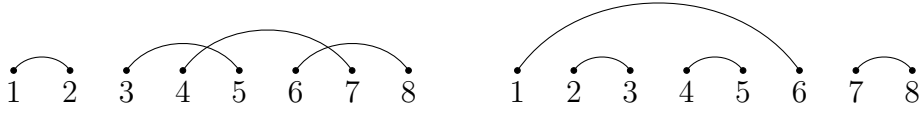
\begin{figure}
  \begin{tikzpicture}[scale=0.75]
    \Matching{1}{2}\Matching{3}{5}\Matching{4}{7}\Matching{6}{8}
    \foreach \i in {1,...,8}{
      \draw [fill] (\i,0) circle [radius=0.05] ;
      \node[below] at(\i,0) {\i};
    }
  \end{tikzpicture} \qquad
  \begin{tikzpicture}[scale=0.75]
    \Matching{1}{6}\Matching{2}{3}\Matching{4}{5}\Matching{7}{8}
    \foreach \i in {1,...,8}{
      \draw [fill] (\i,0) circle [radius=0.05] ;
      \node[below] at(\i,0) {\i};
    }
  \end{tikzpicture}
  \caption{Two matchings on \( [8] \). The first one is nonnesting, while the second one is noncrossing.}
  \label{fig:matching}
\end{figure}
For a matching $M$ and $\{i,j\}\in M$ with $i<j$, $i$ is called an
\emph{opener} and $j$ is called a \emph{closer}. 
Let \( \Mat_{2n} \) (\( \NC_{2n} \) and \( \NN_{2n} \), respectively) stand for
the set of (noncrossing and nonnesting, respectively) matchings on \( [2n] \).

Note that there is a natural bijection between \( \SYT(n,n) \) and
\( \NN_{2n} \). For \( T\in\SYT(n,n) \), connect two vertices lying
on the same column of \( T \) via an arc, then we obtain a
nonnesting matching. For instance, the tableau in Figure~\ref{fig:SYT}
and the first matching in Figure~\ref{fig:matching} are under this correspondence.
Using this correspondence, we index the Specht basis for \( \Spec^{(n,n)} \)
by nonnesting matchings of \( [2n] \), instead of standard Young tableaux of shape
\( (n,n) \):
\[
  \{ v_M \in \Spec^{(n,n)} : M\in\NN_{2n} \}.
\]

We now consider the \( 2 \times 2n \) matrix
\[
  z = 
  \begin{bmatrix}
    z_{1,1} & z_{1,2} & \dots & z_{1,2n} \\
    z_{2,1} & z_{2,2} & \dots & z_{2,2n}
  \end{bmatrix},
\]
where \( z_{i,j} \)'s are indeterminates.
For \( 1\le i<j\le 2n \), let \( \Delta_{ij}:=\Delta_{ij}(z) \) be the maximal minor of \( z \) with respect to the \( i \)th and \( j \)th columns,
i.e., \( \Delta_{ij} = z_{1,i} z_{2,j} - z_{1,j} z_{2,i} \).
For a matching \( M\in\Mat_{2n} \), let
\[
  \Delta_M := \Delta_M(z) = \prod_{\{i,j\}\in M} \Delta_{ij} \in \CC[z_{1,1},\dots,z_{2,2n}].
\]
It is important to note that the polynomials \( \Delta_{ij} \) satisfy the following relation:
For \( 1\le a<b<c<d \le 2n\),
\begin{equation} \label{eq:syzygy}
  \Delta_{ac} \Delta_{bd}
    = \Delta_{ab} \Delta_{cd} + \Delta_{ad} \Delta_{bc}.
\end{equation}
We define a vector space \( W_n \) to be the \( \CC \)-span of \( \Delta_M \)
for all \( M\in\Mat_{2n} \).
In \cite{KR84}, it turns out that the set
\begin{equation} \label{eq:basis_NC}
  \{ \Delta_{M}\in W_n : M\in\NC_{2n} \}
\end{equation}
forms a basis for \( W_n \). We call this basis the \emph{web basis}.
(The web basis was developed in the \( \SL_2 \)-invariant theory
due to Kuperburg~\cite{Kup96},
and its original construction slightly differs from the one we describe above.
But they are essentially the same; see \cite{Rho19}.)

In addition, there is a natural \( \SYM_{2n} \)-action on \( W_n \)
as follows: Regarding a permutation \( \sigma\in\SYM_{2n} \) as
a \( 2n\times 2n \) permutation matrix, define \( \sigma\cdot \Delta_M(z)
:= \Delta_M(z\sigma^{-1}) \).
Then the space \( W_n \) is closed under this action,
and hence carries an \( \SYM_{2n} \)-module structure.
Furthermore, the \( \SYM_{2n} \)-module \( W_n \) is isomorphic to
the Specht module \( \Spec^{(n,n)} \)~\cite{PPR09}.
Therefore, due to Schur's lemma, there is a unique (up to scalar)
isomorphism between \( W_n \) and \( \Spec^{(n,n)} \).

We are now in a position to give the main purpose of this article.
Let \( M_0 \) be the unique matching which is simultaneously noncrossing and
nonnesting, i.e., \( M_0 = \{ \{1,2\},\dots,\{2n-1,2n\} \} \).
Due to \cite{RT19}, the isomorphism maps \( \Delta_{M_0} \) to \( v_{M_0} \) up to scalar.
Let \( \varphi: W_n \rightarrow \Spec^{(n,n)} \) be the unique isomorphism
with \( \varphi(\Delta_{M_0}) = v_{M_0} \).
We also let \( w_M := \varphi(\Delta_M) \) for each \( M\in\NC_{2n} \).
Then the Specht basis can expand into (the image of) the web basis:
For \( M\in\NN_{2n} \),
\[
  v_M = \sum_{M'\in\NC_{2n}} a_{MM'} w_{M'}.
\]
In \cite{RT19}, Russell and Tymoczko initiated the combinatorial study of
the transition matrix
\[
  A = (a_{MM'})_{M\in\NN_{2n}, M'\in\NC_{2n}}.
\]
They constructed directed graphs on the standard
Young tableaux and noncrossing matchings, and using them,
showed the unitriangularity of the matrix.
They also gave some open problems related to their results.
One of them is the positivity of the entries of \( A \),
which was proved by Rhoades soon after.
\begin{thm}[\cite{Rho19}]
  The entries \( a_{MM'} \) of the transition matrix \( A \) are
  nonnegative integers.
\end{thm}
Although Rhoades established the positivity phenomenon for entries of \( A \),
he did not find an explicit combinatorial interpretation of the nonnegative integer \( a_{MM'} \), c.f. \cite[Problem 1.3]{Rho19}.
Inspired by his work, we introduce a new family of permutations which
are enumerated by the integers \( a_{MM'} \), and study their enumerative
properties.

Our strategy is based on Rhoades's observation~\cite{Rho19}.
He figured out that the entries \( a_{MM'} \) are related to resolving crossings
of matchings in the following sense:
For a matching \( M\in\Mat_{2n} \), let \( \{ a,c \} \) and \( \{ b,d \} \) be
a crossing pair in \( M \) (if it exists) where \( a<b<c<d \).
Let \( M' \) and \( M'' \) be the matchings identical to \( M \) except that
\( \{ a,b \} \) and \( \{ c,d \} \) in \( M' \), and
\( \{ a,d \} \) and \( \{ b,c \} \) in \( M'' \).
Then, by the relation \eqref{eq:syzygy}, we have
\begin{equation}\label{Eq: web relation}
  \Delta_M = \Delta_{M'} + \Delta_{M''}.
\end{equation}
In addition, the number of crossing pairs in \( M' \) (respectively, \( M'' \)) is
strictly less than the number of crossing pairs in \( M \).
Therefore, iterating the resolving procedure gives the expansion of \( \Delta_M \)
in terms of the basis \eqref{eq:basis_NC}.
In other words, when we write 
\begin{equation} \label{eq: web expansion}
  \Delta_M = \sum_{M'\in\NC_{2n}}c_{MM'} \Delta_{M'},
\end{equation}
the coefficient \( c_{MM'} \) is equal to
the number of occurrences of the noncrossing matching \( M' \)
obtained by iteratively resolving crossings in \( M \).
Note that the order of the choice of crossing pairs does not affect the expansion of \( \Delta_M \).
Rhoades showed that for \( M\in\NN_{2n} \) and \( M'\in\NC_{2n} \),
the entry \( a_{MM'} \) of the transition matrix equals \( c_{MM'} \).
Hence, to give a combinatorial interpretation of \( a_{MM'} \), we track
the resolving process from a nonnesting matching to noncrossing matchings.

To state our main result, we need some preliminaries. First, we note that
noncrossing matchings and nonnesting matchings are \emph{Catalan objects}, that is, they are enumerated by Catalan numbers.
Another famous Catalan object is a Dyck path.
A \emph{Dyck path} of length \( 2n \)  is a lattice path from \( (0,0) \)
to \( (n,n) \) consisting of \( n \) north steps \( (1,0) \) and
\( n \) east steps \( (0,1) \) that does not pass below the line \( y=x \).
We write \( \NS \) and \( \ES \) for the north step and the east step,
respectively.
We therefore regard a Dyck path as a sequence consisting of
\( n \) \( \NS \)'s and \( n \) \( \ES \)'s.
Let \( \Dyck_{2n} \) be the set of Dyck paths of length \( 2n \).
Identifying a Dyck path with the region below the path,
we give a natural partial order on \( \Dyck_{2n} \) by inclusion,
denoted by \( \subseteq \).
For instance, the Dyck path \( \NS\cdots\NS\ES\cdots\ES \) where
\( n \) \( \NS \)'s precede \( n \) \( \ES \)'s is the maximum path
in \( \Dyck_{2n} \) with respect to the partial order,
while the path \( \NS\ES\NS\ES\cdots\NS\ES \) is the minimum path.
In Section~\ref{sec:grid configuration}, we define a map
\( D:\Mat_{2n}\rightarrow \Dyck_{2n} \), and by abuse of notation,
a map \( D:\SYM_n\rightarrow \Dyck_{2n} \).
We also define a map \( M:\SYM_n\rightarrow \NC_{2n} \).
Finally, we introduce a new family of permutations,
called \emph{web permutations}. With these data, we now present our main result.
\begin{thm} \label{thm:main_intro}
  For matchings \( M\in\NN_{2n} \) and \( M'\in\NC_{2n} \),
  the entry \( a_{MM'} \) is equal to the number of web permutations
  \( \sigma\in\SYM_n \) such that \( D(\sigma)\subseteq D(M) \) and \(M(\sigma)=M'\).
\end{thm}

The theorem follows almost immediately from the definition of the novel
permutations.
However, the definition does not directly tell us whether a given permutation
is a web permutation or not.
In Theorem~\ref{thm:web=andre_cycle}, we thus explain how to characterize
these permutations in terms of their cycle structures.
Using this characterization, we deduce the results in \cite{RT19,IZ21}
concerning the unitriangularity of the transition matrix and a necessary
and sufficient condition for additional vanishing entries.

The article is organized as follows.
In Section~\ref{sec:grid configuration}, we give a new model,
called a grid configuration, for representing matchings.
Within this model, we resolve crossings in nonnesting matchings
until there is no crossing.
We then define web permutations from the noncrossing grid configurations,
and prove the main theorem.
In the next two sections, we study some properties of web permutations.
In Section~\ref{sec:characterization}, we give a characterization of
web permutations. We show that web permutations are closely related to
Andr\'e permutations.
Section~\ref{sec:enumeration} provides some interesting enumerative properties
of web permutations.
One instance of them is that web permutations are enumerated by Euler numbers.
We also give a conjecture for a relation between certain web permutations
and the Seidel triangle.
In Appendix~\ref{sec:appen}, we give some computational data of the transition
matrix and web permutations for small \( n \).

\section{Grid configurations and web permutations}
\label{sec:grid configuration}
In this section, we define grid configurations which represent matchings
in a `rigid' setting.
We describe the procedure of resolving crossings within this model.
We then introduce a new class of permutations, called web permutations.
This provides a combinatorial interpretation for the entries \( a_{MM'} \) of the transition
matrix.

Consider an $n$ by $n$ (lattice) grid in the $xy$-plane with corners
$(0,0), (0,n), (n,0)$ and $(n,n)$.
We denote each cell by $(i,j)$ where $i$ and $j$ are the $x$- and $y$-coordinates of its upper-right corner.
Let $\sigma\in\SYM_n$ be a permutation. For each $1\le i\le n$, mark the cell \( (i,\sigma(i)) \),
and draw a horizontal line to the left and a vertical line to the top from the marked cell.
We call this the \emph{empty grid configuration} of \( \sigma \).
A cell $(i,j)$ is a \emph{crossing} if
there are both a vertical line and a horizontal line through the cell, 
that is, \( \sigma(i) < j \) and \( i < \sigma^{-1}(j) \).
We denote by \( \Cr(\sigma) \) the set of all crossings of \( \sigma \).
For a subset \( E\subseteq\Cr(\sigma) \),
the \emph{grid configuration} \( G(\sigma,E) \) of a pair \( (\sigma, E) \) is defined to be
the empty grid configuration of \( \sigma \) where each crossing in \( E \)
is replaced by an elbow as shown in Figure~\ref{fig:A crossing to an elbow}.
In particular, the empty grid
configuration of \( \sigma \) is \( G(\sigma,\emptyset) \).
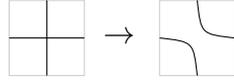
\begin{figure}
  \centering
  \begin{tikzpicture}[scale=1]
  \draw[black!20] (0,0) grid (1,1);
  \Cross{1}{1}
  \node[text width=0.5cm] at (1.5, 0.5) {$\rightarrow$};
  \draw[black!20] (2,0) grid (3,1);
  \Asmooth{3}{1}
  \end{tikzpicture}
  \caption{A crossing to an elbow.}
  \label{fig:A crossing to an elbow}
\end{figure}

For the \( n \) by \( n \) grid, we label leftmost vertical intervals
from bottom to top with 1 through $n$ and uppermost horizontal intervals
from left to right with $n+1$ through $2n$.
With this label of boundary intervals, a grid configuration can be considered
as a matching on \( [2n] \) as follows:
Each strand joining $i$th and $j$th boundary intervals represents
an arc connecting $i$ and $j$;
see Figure~\ref{Fig: The grid configuration G(1324, (1,3),(1,4))}.
We denote by $M(\sigma,E)$ the matching associated to the grid configuration
\( G(\sigma,E) \). For short, we write \( M(\sigma) = M(\sigma, \Cr(\sigma)) \).

\begin{figure}
  \centering
  \begin{tikzpicture}[scale=0.7]
      \node[left] at (0,-1.5) {1};
  \node[left] at (0,-0.5) {2};
  \node[left] at (0,0.5) {3};
  \node[left] at (0,1.5) {4};
  \node[above] at (0.5,2) {5};
  \node[above] at (1.5,2) {6};
  \node[above] at (2.5,2) {7};
  \node[above] at (3.5,2) {8};
  \draw[black!20] (0,-2) grid (4,2);
  \Xmarking{4}{2}\Xmarking{2}{1}\Xmarking{3}{0}\Xmarking{1}{-1}

  \Cross{1}{0}\Cross{2}{2}
  \UP{3}{1}
  \EAST{2}{0}
  \Asmooth{1}{2} \Asmooth{1}{1}
  \Cross{3}{2}

  \Matching{6}{8}\Matching{7}{12}\Matching{9}{11}\Matching{10}{13}
  
  \foreach \i in {1,...,8}{
    \draw [fill] (\i+5,0) circle [radius=0.05] ;
    \node[below] at(\i+5,0) {\i};
  }
  \end{tikzpicture}
  \caption{The grid configuration $G(1324,\{(1,3),(1,4)\})$ and the corresponding matching.}
  \label{Fig: The grid configuration G(1324, (1,3),(1,4))}
\end{figure}
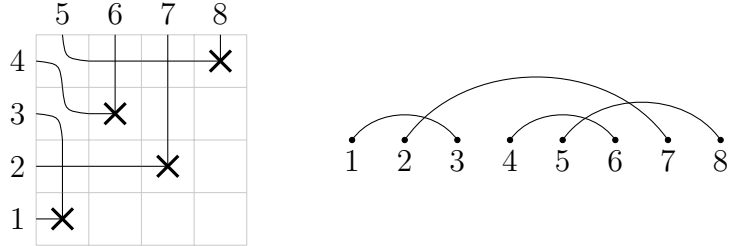

We define a partial order on cells of the $n$ by $n$ grid
by $(x,y)\succeq (x',y')$ if $x\le x'$ and $y\ge y'$.
In other words, \( (x,y)\succeq (x',y') \) if and only if the cell \( (x,y) \) lies
on the upper-left quadrant at \( (x', y') \).

The relation~\eqref{Eq: web relation} can be interpreted as a relation between grid configurations as follows.
For a permutation \( \sigma \) and \( E\subseteq\Cr(\sigma) \),
let $c=(i,j)$ be a maximal crossing in the grid configuration $G(\sigma,E)$,
i.e., there is no crossing on the upper-left quadrant at $c$.
One way of resolving $c$ results a grid configuration $G(\sigma, E\cup \{c\})$.
This procedure of resolving a crossing is called \emph{smoothing}.
The other way of resolving $c$ results a grid configuration $G(\sigma', E)$,
where $\sigma'$ is defined by 
\[
  \begin{cases}
  \sigma'(i)=j,\\
  \sigma'(\sigma^{-1}(j))=\sigma(i), \text{ and}\\
  \sigma'(k)=\sigma(k) \text{ for } k\neq i,\sigma^{-1}(j).
  \end{cases}
\]
This procedure of resolving a crossing is called \emph{switching}.
Note that the crossing sets $\Cr(\sigma)$ and $\Cr(\sigma')$ are not the same. 
Nevertheless, by choosing $c$ to be maximal, crossings not smaller than $c$
(with respect to the partial order) are left unchanged under switching.
In particular, we have \( E\subseteq \Cr(\sigma') \), so switching is well-defined.
We often consider a grid configuration $G$ as the vector $\Delta_{M(G)}$.
Therefore, we can write the relation \eqref{Eq: web relation}
in terms of grid configurations as
\[
  G(\sigma,E)=G(\sigma,E\cup\{c\})+G(\sigma',E).
\]
For example, let \( \sigma=1324\in\SYM_4 \) and \( E=\{(1,3),(1,4)\} \), and
consider the grid configuration \( G(\sigma, E) \) which is shown in
Figure~\ref{Fig: The grid configuration G(1324, (1,3),(1,4))}.
Resolving a maximal crossing  \( c=(2,4)\in\Cr(\sigma)\setminus E \),
we have
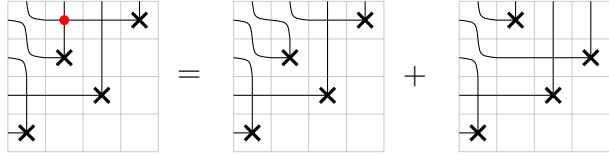
\begin{center}
  \begin{tikzpicture}[scale=0.5]
    \draw[black!20] (0,0) grid (4,4);
    \Xmarking{4}{4}\Xmarking{2}{3}\Xmarking{3}{2}\Xmarking{1}{1}
    \Cross{1}{2}\Cross{2}{4}
    \UP{3}{3}
    \EAST{2}{2}
    \Asmooth{1}{4} \Asmooth{1}{3}
    \Cross{3}{4}

    \node[circle, fill=red, draw=red, scale=0.3] at (1.5,3.5) {};
    \node[text width=0.5cm] at (5, 2) {$=$};

    \draw[black!20] (6,0) grid (10,4);
    \Xmarking{10}{4}\Xmarking{8}{3}\Xmarking{9}{2}\Xmarking{7}{1}

    \Cross{7}{2}\Asmooth{8}{4}
    \UP{9}{3}
    \EAST{8}{2}
    \Asmooth{7}{4} \Asmooth{7}{3}
    \Cross{9}{4}
      \node[text width=0.5cm] at (11, 2) {$+$};
    \draw[black!20] (12,0) grid (16,4);
    \Xmarking{16}{3}\Xmarking{14}{4}\Xmarking{15}{2}\Xmarking{13}{1}

    \Cross{13}{2}
    \UP{15}{3}\UP{16}{4}
    \EAST{14}{2}
    \EAST{14}{3} \EAST{15}{3}
    \Asmooth{13}{4} \Asmooth{13}{3}
    \UP{15}{4}
  \end{tikzpicture}.
\end{center}
Here, the red dot indicates the crossing \( c \).

From the grid configuration $G(id, \emptyset)$, we obtain two grid
configurations by resolving a crossing by smoothing and switching, respectively.
By resolving crossings until there is no crossing left,
we get grid configurations of the form $G(\sigma, \Cr(\sigma))$.
For each remaining grid configuration $G(\sigma, \Cr(\sigma))$,
the permutation $\sigma$ is called a \emph{web permutation} of $[n]$ and
we denote the set of web permutations of $[n]$ by $\Web_n$.
In other words, we have
\begin{equation}\label{Equation: empty grid=sum of grid}
  G(id, \emptyset) = \sum G(\sigma, \Cr(\sigma)),
\end{equation}
where the right hand side is the sum of all grid configurations obtained
by resolving crossings from the grid configuration \( G(id,\emptyset) \)
until there is no crossing left.
This is reminiscent of \eqref{eq: web expansion}.
For example, starting from the grid configuration $G(id,\emptyset)$ for $n=3$,
we have 
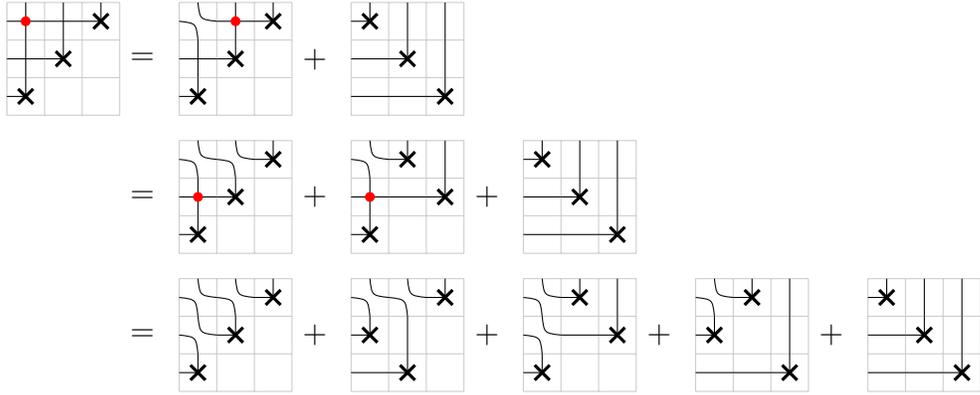
\begin{align*}
    \begin{tikzpicture}[scale=0.5]
        \draw[black!20] (0,0) grid (3,3);
        \Xmarking{1}{1}\Xmarking{2}{2}\Xmarking{3}{3}
        \Cross{1}{2}\Cross{1}{3}\Cross{2}{3}
            \node[circle, fill=red, draw=red, scale=0.3] at (0.5,2.5) {};
    \end{tikzpicture}
  &\begin{tikzpicture}[scale=0.5]
        \draw[white!20] (0,0) grid (1,3);
        \node[text width=0.5cm] at (0.5, 1.5) {$=$};
        \end{tikzpicture}
    \begin{tikzpicture}[scale=0.5]
        \draw[black!20] (0,0) grid (3,3);
        \Xmarking{1}{1}\Xmarking{2}{2}\Xmarking{3}{3}
        \Cross{1}{2}\Asmooth{1}{3}\Cross{2}{3}
            \node[circle, fill=red, draw=red, scale=0.3] at (1.5,2.5) {};
    \end{tikzpicture}
    \begin{tikzpicture}[scale=0.5]
        \draw[white!20] (0,0) grid (1,3);
        \node[text width=0.5cm] at (0.5, 1.5) {$+$};
        \end{tikzpicture}
    \begin{tikzpicture}[scale=0.5]
        \draw[black!20] (0,0) grid (3,3);
        \Xmarking{1}{3}\Xmarking{2}{2}\Xmarking{3}{1}
        \UP{2}{3}\EAST{1}{2}\UP{3}{2}\UP{3}{3}\EAST{1}{1}\EAST{2}{1}
    \end{tikzpicture}\\
    &\begin{tikzpicture}[scale=0.5]
        \draw[white!20] (0,0) grid (1,3);
        \node[text width=0.5cm] at (0.5, 1.5) {$=$};
        \end{tikzpicture}\begin{tikzpicture}[scale=0.5]
        \draw[black!20] (0,0) grid (3,3);
        \Xmarking{1}{1}\Xmarking{2}{2}\Xmarking{3}{3}
        \Cross{1}{2}\Asmooth{1}{3}\Asmooth{2}{3}
            \node[circle, fill=red, draw=red, scale=0.3] at (0.5,1.5) {};
    \end{tikzpicture}
        \begin{tikzpicture}[scale=0.5]
        \draw[white!20] (0,0) grid (1,3);
        \node[text width=0.5cm] at (0.5, 1.5) {$+$};
        \end{tikzpicture}
            \begin{tikzpicture}[scale=0.5]
        \draw[black!20] (0,0) grid (3,3);
        \Xmarking{1}{1}\Xmarking{2}{3}\Xmarking{3}{2}
        \Cross{1}{2}\Asmooth{1}{3}\EAST{2}{2}\UP{3}{3}
                    \node[circle, fill=red, draw=red, scale=0.3] at (0.5,1.5) {};
    \end{tikzpicture}
        \begin{tikzpicture}[scale=0.5]
        \draw[white!20] (0,0) grid (1,3);
        \node[text width=0.5cm] at (0.5, 1.5) {$+$};
        \end{tikzpicture}
    \begin{tikzpicture}[scale=0.5]
        \draw[black!20] (0,0) grid (3,3);
        \Xmarking{1}{3}\Xmarking{2}{2}\Xmarking{3}{1}
        \UP{2}{3}\EAST{1}{2}\UP{3}{2}\UP{3}{3}\EAST{1}{1}\EAST{2}{1}
    \end{tikzpicture}\\
    &\begin{tikzpicture}[scale=0.5]
        \draw[white!20] (0,0) grid (1,3);
        \node[text width=0.5cm] at (0.5, 1.5) {$=$};
        \end{tikzpicture}    \begin{tikzpicture}[scale=0.5]
        \draw[black!20] (0,0) grid (3,3);
        \Xmarking{1}{1}\Xmarking{2}{2}\Xmarking{3}{3}
        \Asmooth{1}{2}\Asmooth{1}{3}\Asmooth{2}{3}
    \end{tikzpicture}
        \begin{tikzpicture}[scale=0.5]
        \draw[white!20] (0,0) grid (1,3);
        \node[text width=0.5cm] at (0.5, 1.5) {$+$};
        \end{tikzpicture}
    \begin{tikzpicture}[scale=0.5]
        \draw[black!20] (0,0) grid (3,3);
        \Xmarking{1}{2}\Xmarking{2}{1}\Xmarking{3}{3}
        \Asmooth{1}{3}\Asmooth{2}{3}
        \EAST{1}{1}\UP{2}{2}
    \end{tikzpicture}
        \begin{tikzpicture}[scale=0.5]
        \draw[white!20] (0,0) grid (1,3);
        \node[text width=0.5cm] at (0.5, 1.5) {$+$};
        \end{tikzpicture}
    \begin{tikzpicture}[scale=0.5]
        \draw[black!20] (0,0) grid (3,3);
        \Xmarking{1}{1}\Xmarking{2}{3}\Xmarking{3}{2}
        \Asmooth{1}{2}\Asmooth{1}{3}\EAST{2}{2}\UP{3}{3}
    \end{tikzpicture}
        \begin{tikzpicture}[scale=0.5]
        \draw[white!20] (0,0) grid (1,3);
        \node[text width=0.5cm] at (0.5, 1.5) {$+$};
        \end{tikzpicture}
    \begin{tikzpicture}[scale=0.5]
        \draw[black!20] (0,0) grid (3,3);
        \Xmarking{1}{2}\Xmarking{2}{3}\Xmarking{3}{1}
        \Asmooth{1}{3}\EAST{2}{1}\EAST{1}{1}\UP{3}{3}\UP{3}{2}
    \end{tikzpicture}
        \begin{tikzpicture}[scale=0.5]
        \draw[white!20] (0,0) grid (1,3);
        \node[text width=0.5cm] at (0.5, 1.5) {$+$};
        \end{tikzpicture}
    \begin{tikzpicture}[scale=0.5]
        \draw[black!20] (0,0) grid (3,3);
        \Xmarking{1}{3}\Xmarking{2}{2}\Xmarking{3}{1}
        \UP{2}{3}\EAST{1}{2}\UP{3}{2}\UP{3}{3}\EAST{1}{1}\EAST{2}{1}
    \end{tikzpicture}.  
  \end{align*}
Therefore we conclude that $\Web_3=\{123, 213, 132, 231, 321\}$.
The following proposition justifies that web permutations are well-defined. 

\begin{prop} \label{prop:well-def}
  The expansion in \eqref{Equation: empty grid=sum of grid} is unique.
  In other words, the grid configurations appearing in
  \eqref{Equation: empty grid=sum of grid} does not depend on
  the order of resolving procedure (choice of maximal crossings).
  In addition, the permutations $\sigma$ in
  \eqref{Equation: empty grid=sum of grid} are all distinct. 
\end{prop}
\begin{proof}
  Any total order extending the partial order \( \succeq \) on cells can be
  obtained from another total order by applying a sequence of changing
  the order of two incomparable cells.
  Therefore it suffices to show that we can change the order of
  two maximal crossings.
  Let $c$ and $c'$ be two maximal crossings in a grid configuration
  $G(\sigma,E)$ with the $x$ coordinate of $c$ is less
  than the $x$-coordinate of $c'$.
  There are two cases: The $y$-coordinate of $c$ and $x$-coordinate of $c'$
  are the same, or not.
  Two such cases are depicted in Figure~\ref{Fig: two cases of maximal crossings},
  where the crossings $c$ and $c'$ are indicated by red dots.

  For the first case, if we resolve both $c$ and $c'$ in the same way
  (both by smoothing or both by switching), the order of resolving $c$ and $c'$
  is irrelevant. Therefore, it remains to show that if we resolve $c$ in a way and
  $c'$ in the other way results the same grid configuration when we resolve
  $c'$ first and then $c$, which can be checked directly.
  In addition, it is clear that the order of resolving
  crossings $c$ and $c'$ is irrelevant for the second case. 

  Let $G(\sigma, E)$ be a grid configuration and $c=(i,j)$ be a maximal crossing
  in $G(\sigma,E)$.
  Suppose that we resolve $c$ by smoothing and then resolve other crossings
  until there is no crossing to obtain a grid configuration of
  the form $G(\tau, \Cr(\tau))$. Since there is an elbow at $c$,
  we have $\tau(i)<j$. On the other hand, suppose that we resolve $c$ by switching
  and then resolve other crossings until there is no crossing to obtain a grid
  configuration of the form $G(\rho, \Cr(\rho))$. Since there is a marking at $c$,
  we have $\rho(i)=j$. By this observation, we conclude that web permutations are all distinct.
\end{proof}

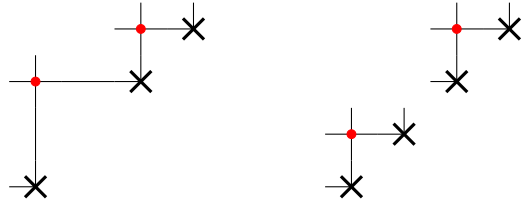
\begin{figure}
  \centering
    \begin{tikzpicture}[scale=0.7]
      \Xmarking{4}{4}\Xmarking{3}{3}\Xmarking{1}{1}
      \Cross{1}{3} \Cross{3}{4}
      \UP{1}{2}
      \EAST{2}{3}
      \node[circle, fill=red, draw=red, scale=0.3] at (0.5,2.5) {};
      \node[circle, fill=red, draw=red, scale=0.3] at (2.5,3.5) {};
  
    \Xmarking{10}{4}\Xmarking{9}{3}\Xmarking{8}{2}\Xmarking{7}{1}
  
      \Cross{7}{2} \Cross{9}{4}
      \node[circle, fill=red, draw=red, scale=0.3] at (6.5,1.5) {};
      \node[circle, fill=red, draw=red, scale=0.3] at (8.5,3.5) {};
  
    \end{tikzpicture}
  
  \caption{Two cases of maximal crossings.}
  \label{Fig: two cases of maximal crossings}
  \end{figure}

\begin{figure}
  \centering
  \begin{tikzpicture}[scale=0.7]
    \draw[black!20] (0,0) grid (5,5);
    \OnlyXmarking{1}{2}\OnlyXmarking{2}{1}\OnlyXmarking{3}{3}\OnlyXmarking{4}{5}\OnlyXmarking{5}{4}
    \draw[line width=0.5mm,blue] (0,0)--(0,2);
    \draw[line width=0.5mm,blue] (0,2)--(2,2);
    \draw[line width=0.5mm,blue] (2,2)--(2,3);
    \draw[line width=0.5mm,blue] (2,3)--(3,3);
    \draw[line width=0.5mm,blue] (3,3)--(3,5);
    \draw[line width=0.5mm,blue] (3,5)--(5,5);
  \end{tikzpicture}
  \caption{The Dyck path \( D(\sigma) \) associated to \( \sigma=21354 \) is
  \( \NS\NS\ES\ES\NS\ES\NS\NS\ES\ES \).}
  \label{fig:D_sigma}
\end{figure}
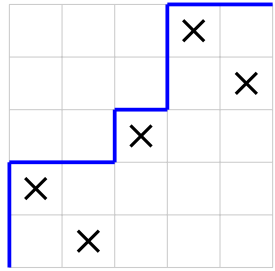

For a matching $M$, record $\NS$ for openers and $\ES$ for closers reading $M$
from left to right.
This gives the Dyck path $D(M)$ in the $n$ by $n$ grid.
It is known that the two restrictions of the map
\( D:\Mat_{2n}\rightarrow \Dyck_{2n} \) to \( \NC_{2n} \) and \( \NN_{2n} \)
are bijections.
To a permutation $\sigma$, we associate the minimum Dyck path \( D(\sigma) \)
where every cell \( (i,\sigma(i)) \) lies below the path;
see Figure~\ref{fig:D_sigma}.

Given a nonnesting matching \( M\in\NN_{2n} \), let \( E(M) \) be the set of
cells in the $n$ by $n$ grid which are above the path $D(M)$.
It is easy to see that the matchings \( M \) and $M(id, E(M))$ coincide.
For example, let \( M \) be the first matching in Figure~\ref{fig:matching}.
Then the corresponding path \( D(M) \) is \( \NS\ES\NS\NS\ES\NS\ES\ES \),
and \( E(M) = \{ (1,2), (1,3), (1,4), (2,4) \} \).
The grid configuration \( G(id, E(M)) \) is shown in Figure~\ref{fig:G(id,E(M))},
and one can see \( M(id, E(M)) = M \).
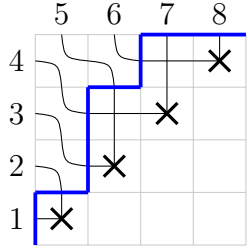
\begin{figure}
  \begin{tikzpicture}[scale=0.7]
    \node[left] at (0,-1.5) {1};
    \node[left] at (0,-0.5) {2};
    \node[left] at (0,0.5) {3};
    \node[left] at (0,1.5) {4};
    \node[above] at (0.5,2) {5};
    \node[above] at (1.5,2) {6};
    \node[above] at (2.5,2) {7};
    \node[above] at (3.5,2) {8};
    \draw[black!20] (0,-2) grid (4,2);
    \Xmarking{4}{2}\Xmarking{3}{1}\Xmarking{2}{0}\Xmarking{1}{-1}
    \Asmooth{1}{0} \Asmooth{1}{1} \Asmooth{1}{2} \Asmooth{2}{2}
    \Cross{2}{1} \Cross{3}{2}
    \draw[line width=0.5mm,blue] (0,-2)--(0,-1);
    \draw[line width=0.5mm,blue] (0,-1)--(1,-1);
    \draw[line width=0.5mm,blue] (1,-1)--(1,1);
    \draw[line width=0.5mm,blue] (1,1)--(2,1);
    \draw[line width=0.5mm,blue] (2,1)--(2,2);
    \draw[line width=0.5mm,blue] (2,2)--(4,2);
  \end{tikzpicture}
  \caption{The grid configuration \( G(id, E(M)) \) and the Dyck path \( D(M) \) where
            \( M=\{ \{1,2\},\{3,5\},\{4,7\},\{6,8\} \} \).}
  \label{fig:G(id,E(M))}
\end{figure}
Similarly to the definition of \( \Web_n \), we consider the equation
\[
  G(id,E(M)) = \sum G(\sigma,\Cr(\sigma)),
\]
where the right hand side is the summation of grid configurations obtained by resolving
crossings in \( G(id,E(M)) \) until there is no crossing.
We then define \( \Web_M \) to be
the set of permutations \( \sigma \) appearing in the right hand side of the above equation.
In particular, \( \Web_n = \Web_M \) where
\( M=\{ \{1,n+1\}, \{2,n+2\},\dots, \{n,2n\} \}. \)

Using the above notations, we prove one of our main results that tells us which
web permutations contribute to the entry $a_{MM'}$.
\begin{proof}[Proof of Thereom~\ref{thm:main_intro}]
  By the definition of web permutations, we have 
  \begin{align*}
    a_{MM'} = | \{ \sigma\in \Web_M : M(\sigma) = M' \} |.
  \end{align*}
  Hence it is enough to show that
  \begin{equation}\label{eq:web(M)}
    \Web_M = \{ \sigma\in \Web_n : D(\sigma)\subseteq D(M) \}.
  \end{equation}
  We can obtain the grid configuration \( G(id,E(M)) \) from \( G(id,\emptyset) \) by smoothing crossings in \( E(M) \).
  Since Proposition~\ref{prop:well-def} says that \( \Web_n \) does not depend on the order of resolving processes, we obtain \( \Web_M \subseteq \Web_n \).
  From this, it is clear that
  \begin{align*}
    \Web_M &= \{ \sigma\in \Web_n : E(M)\subseteq \Cr(\sigma)\}  \\
    &= \{ \sigma\in \Web_n : (i,\sigma(i)) \not\in E(M) \mbox{ for all } i \},
  \end{align*}
  which proves the claim \eqref{eq:web(M)}.
\end{proof}

\section{Characterization of web permutations}
\label{sec:characterization}
In the previous section, we have introduced the new class of permutations
which are obtained by tracking the resolving process.
In fact, Theorem~\ref{thm:main_intro} is just a byproduct of the definition of
web permutations.
In this section, we provide a characterization of these permutations.
This characterization depends only on their permutation structure.
Using this characterization, we also prove the results in \cite{RT19,IZ21}.

We begin with recalling two ways to represent permutations.
One way is the \emph{one-line notation} which we have already used,
that is, regarding a permutation as a word.
More precisely, for a permutation \( \sigma:[n]\rightarrow[n] \),
we write \( \sigma =  \sigma_1 \sigma_2 \dots \sigma_n \)
where \( \sigma_i = \sigma(i) \).
Another way to write permutations is the \emph{cycle notation}.
Instead of the precise definition of this notation, we give an example;
for the definition, see \cite{Sta12}.
Let \( \sigma = 564132\in\SYM_6 \), then the cycle notation of \( \sigma \) is \( (1,5,3,4)(2,6) \).
We always use parentheses and commas for writing cycles.

To describe our characterization of web permutations, we review the notion of
Andr\'e permutations and define an analogue of them.
Andr\'e permutations were introduced by Foata and Sch\"utzenberger~\cite{FSch73},
and have been studied with several applications in the literature, see, e.g., \cite{Sta94, FH16}.
One of the interesting properties of them is that they are enumerated by Euler numbers;
see Section~\ref{sec:enumeration}.

We now think of permutations as words consisting of distinct positive integers.
\emph{Andr\'e permutations} are defined recursively as follows.
First, the empty word and each one-letter word are Andr\'e permutations.
For a permutation \( w=w_1 w_2 \cdots w_n \) with \( n\ge 2 \), 
let \( w_k \) be the smallest letter in \( w \). Then \( w \) is an Andr\'e
permutation if both \( w_1 \cdots w_{k-1} \) and \( w_{k+1}\cdots w_{n} \) are 
Andr\'e permutations and \( \max\{w_1,\dots,w_{k-1}\} < \max\{w_{k+1},\dots,w_n\} \).
For example, a word 547239 is an Andr\'e permutation because
the letter 2 is smallest, both two words 547 and 39 are Andr\'e permutations,
and \( \max\{5,4,7\}<\max\{3,9\} \).
Using this notion, we define a cycle analogue of Andr\'e permutations.
\begin{defn} \label{def:Andre cycle}
  Let \( C=(a_1,\dots,a_k) \) be a cycle with \( a_1 = \min \{a_1,\dots,a_k\} \).
  We say that \( C \) is an \emph{Andr\'e cycle}
  if the permutation \( a_2\cdots a_k \) is an Andr\'e permutation.
\end{defn}
For instance, a cycle \( C = (2,3,9,1,5,4,7) \) is an Andr\'e cycle
since \( C = (1,5,4,7,2,3,9) \) and the permutation 547239 is an Andr\'e permutation.

For a cycle \( C=(a_1,\dots,a_k) \), we write \( \min C = \min \{a_1,\dots,a_k\} \)
and \( \max C = \max \{a_1,\dots,a_k\}\) for short.
The following lemma is useful in the sequel.
\begin{lem} \label{lem:cycle_minmax}
  Let \( C=(a_1,\dots,a_k) \) be an Andr\'e cycle with \( a_1=\min C \).
  Then \( a_k=\max C \).
\end{lem}
\begin{proof}
  By definition, the last letter of an Andr\'e permutation is the largest element
  in the permutation. This fact directly gives the proof.
\end{proof}
The following lemma gives how to obtain a new Andr\'e cycle from old Andr\'e cycles.
\begin{lem} \label{lem:merging}
  Let \( C_1=(a_1,\dots,a_k) \) and \( C_2=(b_1,\dots,b_\ell) \) be Andr\'e cycles with
  \( a_1 = \min C_1 \) and \( b_1=\min C_2 \).
  If \( a_1<b_1 \) and \( a_k<b_\ell \), then the cycle \( (a_1,\dots,a_k, b_1,\dots,b_\ell) \) is also an Andr\'e cycle.
\end{lem}
\begin{proof}
  We induct on \( k \). First, consider the base case \( k=1 \).
  Since \( a_1 = \min\{a_1,b_1,\dots,b_\ell\} \), we only need to show that the
  permutation \( b_1\cdots b_k \) is an Andr\'e permutation.
  This follows immediately from the definition of Andr\'e permutations.
  
  We now suppose \( k\ge 2 \).
  Recall that the two permutations \( a_2\cdots a_k \) and \( b_2\cdots b_\ell \)
  are Andr\'e permutations. In addition, by Lemma~\ref{lem:cycle_minmax} and
  the assumption \( a_k< b_\ell \),
  we have \( \max\{a_2,\dots,a_k\}<\max\{b_2,\dots,b_\ell\} \).
  Thus, if \( b_1< \min\{a_2,\dots,a_k\} \),
  then the permutation \( a_2\cdots a_k b_1\cdots b_\ell \) is an Andr\'e permutation.
  Otherwise, let \( a_p=\min\{a_2,\dots,a_k\} \) for some \( p \), so that \( a_p < b_1 \)
  and both \( a_2\cdots a_{p-1} \) and \( a_{p+1}\cdots a_k \) are Andr\'e permutations.
  By the induction hypothesis, we have that
  the cycle \( (a_p,\dots,a_k,b_1,\dots,b_\ell) \) is an Andr\'e cycle.
  It is also clear that \( a_{p-1} < a_k < b_\ell \).
  Again, by the induction hypothesis,
  we deduce that the cycle \( (a_1,\dots,a_k, b_1,\dots,b_\ell) \) is an Andr\'e cycle, which yields the desired result.
\end{proof}

We now show another main result of the article,
which gives a characterization of web permutations.
\begin{thm} \label{thm:web=andre_cycle}
  A permutation \( \sigma \in \SYM_n \) is a web permutation
  if and only if each cycle of \( \sigma \) is an Andr\'e cycle.
\end{thm}
\begin{proof}
  Recall that the web permutations
  do not depend on the order of choices of maximal crossings.
  Hence we fix the following total order on the cells in the \( n \) by \( n \) 
  grid, which completes the partial order,
  and we assume that our resolving process respects this total order:
  For two cells \( (i,j) \) and \( (i',j') \), we let \( (i,j)>(i',j') \)
  if either \( j>j' \), or \( j=j' \) and \( i<i' \).

  We first prove the ``only if'' part.
  Let \( \sigma \) be a web permutation, and
  \[ (G^{(0)}=G(id, \emptyset), G^{(1)}, \dots, G^{(r)}=G(\sigma, \Cr(\sigma))) \]
  be the sequence of grid configurations
  where \( G^{(k)} \) is obtained from \( G^{(k-1)} \) by resolving a single
  crossing for each \( k \), with respect to the total order.
  We write \( G^{(k)} = G(\sigma^{(k)}, E^{(k)}) \).
  Also let \( c^{(k)} \) be the crossing in \( \Cr(\sigma^{(k-1)})\setminus E^{(k-1)} \)
  such that \( G^{(k)} \) is obtained from \( G^{(k-1)} \) by resolving \( c^{(k)} \).

  It is obvious that the identity permutation \( \sigma^{(0)}=id \) consists of
  Andr\'e cycles. We claim that each \( \sigma^{(k)} \) also consists of Andr\'e cycles 
  for \( 1\le k\le r \), in particular, so does \( \sigma \).
  Fix an integer \( 1\le k\le r \).
  We use an inductive argument, so suppose that each cycle of \( \sigma^{(k-1)} \)
  is an Andr\'e cycle.
  If \( G^{(k)} \) is obtained by smoothing the crossing \( c^{(k)} \) in \( G^{(k-1)} \),
  then \( \sigma^{(k-1)}=\sigma^{(k)} \) and thus there is nothing to prove.
  Therefore, we assume that \( G^{(k)} \) is obtained from \( G^{(k-1)} \)
  by switching the crossing \( c^{(k)}=(i,j) \).
  Then 
  \begin{equation} \label{eq:ij_condition}
    \sigma^{(k-1)}(i)<j \qand i<(\sigma^{(k-1)})^{-1}(j).
  \end{equation}
  Let \( C_1, \dots, C_\ell \) be cycles of \( \sigma^{(k-1)} \).
  We first observe that for each \( 1\le p\le \ell \), all entries in \( C_p \)
  except the minimum \( \min C_p \) are greater than \( j \).
  We justify this observation later.
  From this, we have that \( i \) and \( j \) are contained in different cycles
  of \( \sigma^{(k-1)} \). Indeed, if \( i \) and \( j \) lie on the same cycle,
  then \( \sigma^{(k-1)}(i) \) also lies on the cycle, but it is a contradiction to
  \eqref{eq:ij_condition}.
  Without loss of generality, let \( C_1=(a_1,\dots,a_s) \) and
  \( C_2=(b_1,\dots,b_t) \) contain \( i \) and \( j \) respectively
  with \( a_1=\min C_1 \) and \( b_1 = \min C_2 \).
  By the first inequality of \eqref{eq:ij_condition} and the observation,
  \( a_1 = \sigma^{(k-1)}(i) \) and \( b_1 = j \), so \( a_s = i \) and
  \( b_t = (\sigma^{(k-1)})^{-1}(j) \).
  By definition, resolving the crossing \( (i,j) \) by switching merges two cycles \( C_1 \) and
  \( C_2 \) into the cycle \( C=(a_1,\dots,a_s,b_1,\dots,b_t) \),
  and leaves other cycles unchanged.
  It therefore follows from Lemma~\ref{lem:merging} and \eqref{eq:ij_condition}
  that the cycle \( C \) is an Andr\'e cycle.
  Note that \( \min C = a_1 = \sigma^{(k-1)}(i) < j \), and there is no crossing on
  row \( j \) in the grid configuration \( G^{(k)} \).
  Hence the crossing \( c^{(k+1)} \) lies below row \( j \),
  which implies the observation inductively.

  We now prove the ``if'' part. 
  It suffices to show that we obtain any Andr\'e cycle by iterating resolving
  processes to the identity permutation along the total order.
  We induct on the length of an Andr\'e cycle where the base case being trivial.
  Suppose that \( C=(a_1,\dots,a_k) \) is an Andr\'e cycle with \( k\ge 2 \) and
  \( a_1=\min C \). Then by definition, the permutation \( a_2\cdots a_k \) is
  an Andr\'e permutation. Let \( a_p = \min \{a_2,\dots,a_k\} \) for some
  \( p \), so \( a_2\cdots a_{p-1} \) and \( a_{p+1}\cdots a_k \) are also
  Andr\'e permutations. Thus, the cycles \( (a_1,\dots,a_{p-1}) \) and
  \( (a_p, \dots,a_k) \) are Andr\'e cycles.
  By the induction hypothesis, we can obtain the web permutation
  \( \sigma=(a_1,\dots,a_{p-1})(a_p,\dots,a_k) \) by resolving processes.
  More precisely, we can obtain the grid configuration \( G(\sigma, E) \)
  such that for \( 1\le i\le n \) and \( j\le a_p \), \( (i,j)\notin E \).
  Furthermore, one can easily check that the cell \( (a_{p-1}, a_p) \) belongs to
  \( \Cr(\sigma) \), so \( (a_{p-1}, a_p)\in\Cr(\sigma)\setminus E \).
  We then obtain the desired cycle \( C \) by switching the crossing
  \( (a_{p-1}, a_p) \) in the grid configuration \( G(\sigma, E) \),
  which completes the proof.
\end{proof}

As an application of the characterization, we show that the transition matrix
\( (a_{MM'}) \) is unitriangular with respect to a certain order on \( \NN_{2n} \)
and \( \NC_{2n} \), and determine which entries \( a_{MM'} \) vanish.
These are already known due to Russell--Tymoczko~\cite{RT19} and
Im--Zhu~\cite{IZ21}.

Before we give the vanishing condition,
we first show that the set \( \Web_n \) includes a well-studied class of permutations.
For a permutation \( \sigma = \sigma_1\cdots \sigma_n \),
we say that \emph{\( \sigma \) contains a 312-pattern} if there exist three indices
\( 1\le i<j<k\le n \) such that \( \sigma_j<\sigma_k<\sigma_i \).
A permutation is \emph{312-avoiding} if it does not contain a 312-pattern.
Note that 312-avoiding permutations are a Catalan object. Furthermore,
the restriction of \( D:\SYM_n\rightarrow \Dyck_{2n} \) to the set of
312-avoiding permutations of \( [n] \) is a bijection.
\begin{cor}\label{cor:312-avoiding}
  A 312-avoiding permutation is a web permutation.
\end{cor}
\begin{proof}
  By Theorem~\ref{thm:web=andre_cycle}, it suffices to show the following: 
  For a permutation \( \sigma \), if there is a cycle \(C=(a_1,\dots,a_\ell)\) which is not an Andr\'e cycle in \(\sigma\), then \(\sigma\) contains a 312-pattern which consists of $a_i$'s.

  We use induction on the length of \(C\).
  Since any cycle of length less than 3 is an Andr\'e cycle, the base case is when the length of \(C\) is 3. 
  The only case is of the form \((a_1,a_2,a_3)\), where \(a_1<a_3<a_2\). 
  Thus, \(\sigma\) contains a 312-pattern \(a_1<a_3<a_2\).

  Now assume that the length of \(C\) is larger than 3.
  Write \(C=(a_1,a_2,\dots,a_\ell, b_1,b_2,\dots,b_r)\), where \( a_1 \) and \( b_1 \) is the smallest and the second smallest elements of
  \( C \), respectively.
  Then one of the following holds:
  \begin{enumerate}[label=\roman*)]
    \item \(C_1=(a_1,a_2,\dots,a_\ell)\) is not an Andr\'e cycle.
    \item \(C_2=(b_1,b_2,\dots,b_r)\) is not an Andr\'e cycle.
    \item both \(C_1\) and \(C_2\) are Andr\'e cycles and $\max C_1=a_\ell>b_{r}=\max C_2$.
  \end{enumerate}
  
  For the first case, by the induction hypothesis, 
  there exist integers $0\le i,j,k\le\ell$ such that \(a_i<a_j<a_k\) and \(a_{j+1}<a_{k+1}<a_{i+1}\) where the subscripts are interpreted modulo $\ell$.
  Note that \(\sigma_{a_i}=a_{i+1}\) except for \(\sigma_{a_\ell}=b_1\). 
  Since $b_1$ is the second smallest element, replacing $a_1$ with $b_1$ does not change the pattern of \({a_{i+1}}{a_{j+1}}{a_{k+1}}\).
  Therefore, \(\sigma\) contains a 312-pattern as we claimed.
  The second case can be proved similarly to the first case.
  For the last case, $a_1<b_r<a_\ell$ forms a 312-pattern. 
  Indeed, $\sigma_{a_1}=a_2, \sigma_{b_r}=a_1, \sigma_{a_\ell}=b_1$ and $a_1<b_1<a_2$.
\end{proof}

Recall that the set \( \Dyck_{2n} \) has a partial order \( \subseteq \),
and there are bijections \( D \) from \( \NN_{2n} \) and from \( \NC_{2n} \) to
\( \Dyck_{2n} \).
Then the maps \( D \) induce a partial order on \( \NN_{2n} \) and \( \NC_{2n} \).
Furthermore, when we choose a total order on \( \Dyck_{2n} \) that completes
the partial order \( \subseteq \), the maps \( D \) give a total order
on \( \NN_{2n} \) and \( \NC_{2n} \).

\begin{rmk}
  In \cite{RT19}, Russell and Tymoczko defined a directed graph \( \Gamma \) on
  \( \NC_{2n} \), and defined a partial order on \( \NC_{2n} \) using the digraph.
  The graph \( \Gamma \) is an edge-labeled directed graph whose vertex set is
  the set of noncrossing matchings and its labeled edges are given as follows.
  For \( M,M'\in\NC_{2n} \), assign a labeled, directed edge $M \xrightarrow{i} M'$ if both of the following hold:
  \begin{enumerate}[label=\roman*)]
    \item $M$ has arcs $\{j,k\}$ and $\{i,i+1\}$ while $M'$ has arcs $\{j,i\}$ and
    $\{i+1,k\}$ where \( j<i<k \).
    \item Other arcs in $M$ and $M'$ are the same.
  \end{enumerate}
  The graph $\Gamma$ defines a partial order on $\NC_{2n}$ by letting
  \(M\preceq M'\) if there is a directed path from $M'$ to $M$ in \(\Gamma\).
  Russell and Tymoczko also defined a partial order on \(\NN_{2n}\)
  via a well-known bijection between \( \NN_{2n} \) and \( \NC_{2n} \).
  It is straightforward to see that their partial order on \(\NC_{2n}\)
  and \(\NN_{2n}\) coincides with ours.
\end{rmk}

We now take a total order on \( \Dyck_{2n} \) which completes the partial order
\( \subseteq \), and thus we have the induced total order on
\( \NN_{2n} \) and \( \NC_{2n} \).
We assume that orderings of rows and columns of the transition matrix
\( (a_{MM'}) \) are the decreasing orders with respect to the total order
on \( \NN_{2n} \) and \( \NC_{2n} \).
Then the entry \( a_{MM'} \) is on the diagonal if and only if \( D(M)=D(M') \).

We are now ready to prove the unitriangularity of the transition matrix
$(a_{MM'})$ and the conjecture of Russell and Tymoczko
\cite[Conjecture 5.8]{RT19} concerning the condition of the vanishing entries,
which is later proved by Im and Zhu \cite[Theorem 1.1]{IZ21}.

\begin{cor}\label{Cor: unitriangularity}\cite{RT19,IZ21}
  Let \( M\in\NN_{2n} \) and \( M'\in\NC_{2n} \).
  Then \( a_{MM'} > 0 \) if and only if \( D(M') \subseteq D(M) \).
  In particular, the transition matrix $(a_{MM'})$ is upper-triangular.
  Moreover, there are ones along the diagonal of the transition matrix,
  and 312-avoiding permutations contribute to the ones.
\end{cor}
\begin{proof}
  Recall that by the argument in the proof of Theorem~\ref{thm:main_intro}, we have
  \[
      a_{MM'}=\{\sigma\in\Web_M:M(\sigma)=M'\}.
  \]
  We first show that there are ones along the diagonal, i.e.,
  \( a_{MM'}=1 \) if \( D(M)=D(M') \).
  Let $\sigma$ be a permutation in $\Web_M$ satisfying $M(\sigma)=M'$.
  Denote the set of cells above the Dyck path $D(\sigma)$ by $E(\sigma)$.
  We claim that $E(\sigma)=\Cr(\sigma)$.
  Since \(E(\sigma)\subseteq \Cr(\sigma)\) is obvious, suppose that we have $E(\sigma)\subsetneq\Cr(\sigma)$, and let $c$ be a maximal crossing in \( \Cr(\sigma)\setminus E(\sigma) \).
  Note that \[D(M(\sigma,E(\sigma)\cup \{c\}))\subsetneq D(M(\sigma,E(\sigma)))=D(M).\]
  Thus, if we resolve all crossings as smoothing to obtain $G(\sigma,\Cr(\sigma))$, 
  the associated Dyck path $D(M')=D(M(\sigma))$ lies strictly below $D(M)$ 
  which is a contradiction.

  It is well known that 312-avoiding permutations are only permutations satisfying
  the condition \( \Cr(\sigma)=E(\sigma) \) and the map $D:\Web_n\rightarrow \Dyck_{2n}$ is a bijection
  when restricted to 312-avoiding permutations (see \cite[§1.2]{Sta12}).
  Here, the restriction makes sense by Corollary~\ref{cor:312-avoiding}.
  Combining these facts, it follows that each 312-avoiding permutation represents 
  each one on the diagonal in the transition matrix.

  To show the ``only if'' part of the first assertion, assume that $D(M') \nsubseteq D(M)$.
  Then there exists a cell below the Dyck path $D(M')$ and above the Dyck path
  $D(M)$, i.e., \( E(M)\setminus E(M')\neq\emptyset. \)
  Choose a maximal cell $c=(i,j)$ in \( E(M)\setminus E(M')\neq\emptyset. \)
  Let $\sigma$ be a web permutation. 
  If $\sigma$ is in $\Web_M$, then we have $c\in\Cr(\sigma)$, thus $\sigma_i<j$. 
  On the other hand, if $M(\sigma)=M'$, then we have $\sigma_i=j$, which is a contradiction. 
  Therefore, we have \( a_{MM'}=0 \).

  For the ``if'' part, let $M''$ be the nonnesting matching such that $D(M'')=D(M')$. Then we have
  \begin{align*}
    a_{MM'} 
    &= |\{\sigma \in \Web_n: D(\sigma)\subseteq D(M), M(\sigma)=M'\}| \\
    &\ge |\{\sigma \in \Web_n: D(\sigma)\subseteq D(M')=D(M''), M(\sigma)=M'\}| \\
    &= a_{M''M'}=1.
  \end{align*}
  This completes the proof.
\end{proof}

\section{Enumeration of web permutations}
\label{sec:enumeration}
In this section, we focus on the number of web permutations. 
More precisely, we give a relation between web permutations and
Andr\'e cycles~(Theorem~\ref{thm:webn=Andren+2}),
and show that the numbers of web permutations equal Euler numbers.
We also conjecture that the Seidel triangle can be recovered completely from
the certain classes of web permutations.

We have characterized web permutations using Andr\'e cycles (Theorem~\ref{thm:web=andre_cycle}).
We now present another relationship between web permutations and Andr\'e cycles.
Let us first review the Foata transformation \( \widehat{~~}:\SYM_n\rightarrow\SYM_n \).
For a permutation $\sigma\in\SYM_n$, the \emph{canonical} cycle notation
of \( \sigma \) is a cycle notation of \( \sigma \) such that its cycles are
sorted based on the smallest elements of the cycles and
the smallest element of each cycle is written in the last place of the cycle.
We define \( \widehat{\sigma} \) to be the permutation obtained by dropping 
the parentheses in the canonical cycle notation of \( \sigma \).
A \emph{right-to-left minimum} is an element \( \sigma_i \) such that
\( \sigma_i < \sigma_j \) for all \( j>i \).
Using right-to-left minima of \( \sigma \), one can easily construct the inverse
of the Foata transformation.
Note that the number of cycles of \( \sigma \) equals the number of right-to-left minima of \( \widehat{\sigma} \).

We now introduce a map $\phi : \SYM_n \rightarrow \SYM_{n+2}$ as
a slightly modification of the Foata transformation.
For a permutation $\sigma \in \SYM_n$,
define the one-cycle permutation $\phi(\sigma) \in\SYM_{n+2}$ by
\[
  \phi(\sigma):=(1,\widehat{\sigma}_1 +1, \dots, \widehat{\sigma}_n+1, n+2).
\]
It follows immediately from the bijectivity of the Foata transformation
that the map \( \phi \) is injective, and its image \( \phi(\SYM_n) \) is the
set of one-cycle permutations \( \sigma\in\SYM_{n+2} \) with \( \sigma(n+2)=1 \).
For instance, let \( \sigma=568479312\in\SYM_9 \). In the canonical cycle notation,
\( \sigma=(5,7,3,8,1)(6,9,2)(4) \), so \( \widehat{\sigma}=573816924 \).
Then we have \[ \phi(\sigma) = (1,6,8,4,9,2,7,10,3,5,11)\in\SYM_{11}. \]
The right-to-left minima of \( \widehat{\sigma} \) are \( 1,2 \) and \( 4 \),
which are the minima of cycles of \( \sigma \).
Note that the permutation \( \sigma \) is a web permutation,
and the cycle \( \phi(\sigma) \) is an Andr\'e cycle.
Surprisingly, this is not an accident.
\begin{thm}\label{thm:webn=Andren+2}
  For \( n\ge 1 \), let \( \AC_{n+2}\subset\SYM_{n+2} \) be the set of
  Andr\'e cycles consisting of \( [n+2] \).
  Then we have \( \phi(\Web_n) = \AC_{n+2} \).
  In particular, the number of web permutations of \( [n] \) is equal to
  the number of Andr\'e cycles consisting of \( [n+2] \).
\end{thm}
\begin{proof}
  Let $\sigma$ be a web permutation of \( [n] \).
  In the canonical cycle notation, we write 
  \[
    \sigma=C_1 C_2 \cdots C_k 
  \]
  where \( C_i = (c^{(i)}_{1}, \dots, c^{(i)}_{n_i}) \) is a cycle with $\min C_i = c^{(i)}_{n_i}$ for each \( i=1,\dots,k \), and $1=c^{(1)}_{n_1}<\cdots<c^{(k)}_{n_k}$.
  Note that by Theorem~\ref{thm:web=andre_cycle}, each cycle \( C_i \) is an
  Andr\'e cycle, that is, each word \( c^{(i)}_1\cdots c^{(i)}_{n_i-1} \) is
  an Andr\'e permutation.
  We claim that the word
  \[
    c^{(1)}_1\cdots c^{(1)}_{n_1} c^{(2)}_{1}\cdots c^{(2)}_{n_2}\cdots c^{(k)}_{n_k} (n+1)
  \] obtained by appending \( n+1 \) to the end of \( \widehat{\sigma} \)
  is an Andr\'e permutation.
  Since \( c^{(1)}_{n_1} \) is the minimum in the word, and
  \( c^{(1)}_1\cdots c^{(1)}_{n_1-1} \) is an Andr\'e permutation,
  it suffices to show that the suffix \( c^{(2)}_{1}\cdots c^{(2)}_{n_2}\cdots
  c^{(k)}_{n_k} (n+1) \) is an Andr\'e permutation.
  Then an appropriate inductive argument shows the claim.
  We now consider the cycle \( \phi(\sigma) \).
  Using the canonical cycle notation of \( \sigma \), we have
  \[
    \phi(\sigma)
    = (1,c^{(1)}_1+1,\dots, c^{(1)}_{n_1}+1,c^{(2)}_1 +1,\dots,c^{(2)}_{n_2}+1,\dots,c^{(k)}_1+1 ,\dots, c^{(k)}_{n_k}+1,n+2 ).
  \]
  Thus, by the claim, \(\phi(\sigma)\) is an Andr\'e permutation of \( [n+2] \),
  as desired.

  Conversely, let \( \tau \) be an Andr\'e cycle of \( [n+2] \).
  One can directly check from the definition of Andr\'e cycles that
  \( \tau \) forms \( (1,a_1,\dots,a_n, n+2) \).
  Let \( a_{n_1},\dots,a_{n_k} \) be the right-to-left minima of the permutation
  \( a_1\cdots a_n \) with \( n_1<\dots<n_k \) so that \( a_{n_1}<\dots<a_{n_k} \).
  Then we only need to show that the permutation
  \( (a_1,\dots,a_{n_1})(a_{n_1+1},\dots,a_{n_2})\cdots (a_{n_{k-1}+1},a_{n_k}) \)
  is a web permutation, or equivalently, due to Theorem~\ref{thm:web=andre_cycle}
  each cycle \( (a_{n_{i-1}+1},\dots,a_{n_i}) \) is an Andr\'e cycle.
  It is easily verified by a similar argument as in the previous claim and
  using the right-to-left minima.
  Hence we leave the details to the reader.
\end{proof}

\subsection{Euler and Entringer numbers}
In this subsection, we give various enumerative properties of
web permutations using Theorem~\ref{thm:webn=Andren+2}.

We start with recalling Euler numbers.
The \emph{Euler numbers} \( E_n \) are defined via the exponential
generation function
\[
  E(z) := \sum_{n\ge 0} E_n \frac{x^n}{n!} = \sec z + \tan z.
\]
The first few Euler numbers are 1, 1, 1, 2, 5, 16, 61;
see \cite{OEIS} with ID number A000111.
There are numerous combinatorial objects enumerated by Euler numbers \( E_n \),
e.g., alternating permutations, complete increasing binary trees, and etc.
Especially, the Euler number \( E_n \) counts Andr\'e permutations of \( [n] \).
For details, we refer to \cite{Sta10}, which is a wonderful survey of Euler numbers and related topics.
We provide another occurrence of Euler numbers.

\begin{cor}
  The Euler number \( E_{n+1} \) enumerates
  the number of web permutations of \( [n] \).
\end{cor}
\begin{proof}
  By definition, the number of Andr\'e permutations of \( [n] \) is equal to
  the number of Andr\'e cycles of \( [n+1] \).
  Then Theorem~\ref{thm:webn=Andren+2} implies the desired result.
\end{proof}
\begin{rmk}
  One can prove the corollary without using the fact that the number of Andr\'e
  permutations is equal to the Euler number.
  Indeed, let \( w_n \) be the number of web permutations of \( [n] \),
  \( ac_n \) the number of Andr\'e cycles of \( [n] \), and
  \[
    W(z) = \sum_{n\ge 0} w_n \frac{z^n}{n!}, \qand
    AC(z) = \sum_{n\ge 1} ac_n \frac{z^n}{n!},
  \]
  where we set \( w_0 = 1 \).
  Then by a standard fact of generating functionology \cite[Corollary~5.1.6]{Sta99} and Theorem~\ref{thm:web=andre_cycle},
  we have
  \[
    W(z) = \exp AC(z).
  \]
  Meanwhile, Theorem~\ref{thm:webn=Andren+2} gives the ODE
  \[
    W(z) = \frac{d^2}{dz^2} AC(z) = \frac{d^2}{dz^2} \log W(z)
  \]
  whose unique solution is \( W(z) = \sec z \tan z + \sec^2 z =  E'(z) \), which implies \( w_n = E_{n+1} \).
\end{rmk}

For a permutation \( \sigma \), let \( \cc(\sigma) \) be the number of cycles of
\( \sigma \), and \( \rlmin(\sigma) \) the number of right-to-left minima of
\( \sigma \).
By convention, we set \( \cc(\emptyset) = 0 \) where \( \emptyset \) is
the empty permutation, and \( \Web_0=\{\emptyset\}. \)
Since the Foata transformation gives the equidistribution of the two statistics
\( \cc(\sigma) \) and \( \rlmin(\sigma) \), we have the following corollary
concerning the distribution of \( \cc(\sigma) \) on \( \Web_n \).
\begin{cor} 
  We have
  \begin{align*}
    \left( \frac{1}{1-\sin z} \right)^t =
      \sum_{n\ge 0} \sum_{\sigma\in \Web_n} t^{\cc(\sigma)} \frac{z^{n}}{n!}.
  \end{align*}
\end{cor}
\begin{proof} 
  In \cite[Proposition~1]{Dis13}, the author showed that 
  \begin{align*}
    \left( \frac{1}{1-\sin z} \right)^t = \sum_{n\ge 1}\sum_{\sigma}
    t^{\rlmin(\sigma)-1} \frac{z^{n-1}}{(n-1)! }
  \end{align*}
  where the inner sum is over all Andr\'e permutations of \( [n] \).
  Therefore the proof follows immediately from Theorem~\ref{thm:webn=Andren+2}.
\end{proof}

We also recall Entringer numbers.
The \emph{Entringer} numbers are given by
the generating function
\[
  \frac{\cos x + \sin x}{\cos (x+y)} = \sum_{m,n \ge 0} E_{m+n,[m,n]} \frac{x^m}{m!} \frac{y^n}{n!},
\]
where \( [m,n] \) is \( m \) if \( m+n \) is odd, and \( n \) otherwise.
These numbers refine Euler numbers in the following sense: For \( n\ge 1 \),
\[
  \sum_{k=1}^n E_{n,k} = E_{n+1}.
\]
We have a counterpart of this refinement.
\begin{cor} 
  The Entringer number \( E_{n,k} \) is equal to the number of web permutations \( \sigma \) of \( [n] \) with \( \sigma_1 = n+1-k \).
\end{cor}
\begin{proof} 
  In \cite[Theorem~1.1]{FH16}, the authors showed that \( E_{n,k} \) equals the
  number of Andr\'e permutations \( \sigma \) of \( [n+1] \) with
  \( \sigma_1 = n+1-k \).
  Combining this fact and Theorem~\ref{thm:webn=Andren+2} gives the proof.
\end{proof}

\subsection{Genocchi numbers and the Seidel triangle}
The Genocchi numbers are well-studied numbers with various combinatorial properties; see \cite{Dum74, LW20}. The Genocchi numbers can be defined
by the \emph{Seidel triangle} as follows \cite{Sei77}.
Recall that the Seidel triangle is an array of integers $(s_{i,j})_{i,j\ge1}$
such that $s_{1,1}=s_{2,1}=1$ and
\[
  \begin{cases}
    s_{2i+1,j} = s_{2i+1,j-1} + s_{2i,j} \mbox{ for } j=1,2,\dots, i+1\\
    s_{2i,j} = s_{2i,j+1} + s_{2i-1,j} \mbox{ for } j=i,i-1,\dots,1,
  \end{cases}
\]
where $s_{i,j}=0$ for $j<0$ or $j>\lceil i/2\rceil$.
This Pascal type procedure is called the \emph{boustrophedon algorithm}.
The \emph{Genocchi numbers} $g_n$ are defined by
\[
  g_{2n-1}=s_{2n-1,n} \qand g_{2n}=s_{2n,1}.
\]
In fact, the sequence $(g_n)$ is the interleaving of the Genocchi numbers of
the first kind and the median Genocchi numbers.
The first values of the Seidel triangle and Genocchi numbers (in red)
are given in the following sequence. 

\begin{center}
  \begin{tabular}{c|ccccc}
    \( n\backslash k \) & 1 & 2 & 3 & 4 & 5 \\ \hline
    1 & \textcolor{red}{1} & & & & \\
    2 & \textcolor{red}{1} & & & & \\
    3 & 1 & \textcolor{red}{1} & & & \\
    4 & \textcolor{red}{2} & 1 & & & \\
    5 & 2 & 3 & \textcolor{red}{3} & & \\
    6 & \textcolor{red}{8} & 6 & 3 & & \\ 
    7 & 8 & 14 & 17 & \textcolor{red}{17} & \\
    8 & \textcolor{red}{56} & 48 & 34 & 17 &\\
    9 & 56 & 104 & 138 & 155 & \textcolor{red}{155}\\
  \end{tabular}
\end{center}

Recall that we denote by \( M_0 \) for the unique matching which
is simultaneously noncrossing and nonnesting, i.e.,
\( M_0 = \{ \{1,2\},\dots,\{2n-1,2n\} \} \).
To emphasize the size of the matchings, we denote this unique matching of $[2n]$
by $M^{(n)}_0$. Let $f(n)$ be the number of web permutations $\sigma$ of $[n]$
with $M(\sigma)=M^{(n)}_0$.
In \cite{Nak20}, Nakamigawa showed the following theorem.
\begin{thm}[{\cite[Theorem 3.1]{Nak20}}] \label{thm:Geno}
  For \( n\ge 1 \), we have \( f(n)=g_n \).
\end{thm}
Let $f(n,k)$ be the number of web permutations $\sigma$ of $[n]$ such that
\( M(\sigma) = M^{(n)}_0 \) and $\sigma_1=k$.
Obviously, \( f(n)=\sum_{1\le k\le n} f(n,k) \).
Some of these numbers vanish in the following cases.
\begin{prop}\label{prop:web-Genocchi even k=0}
  For $n\ge1$ and $1\le k \le \lfloor n/2\rfloor$, we have $f(n,2k)=0$.  
\end{prop}
\begin{proof}
  Let $\sigma$ be a web permutation of $[n]$ with $\sigma_1=2k$.
  Then considering the grid configuration \( G(\sigma,\Cr(\sigma)) \),
  the associated matching $M(\sigma)$ has an arc connecting $2k$
  and some $j$ with $2k<j$. Since there is the arc connecting $2k-1$ and $2k$
  in $M^{(n)}_0$, we deduce $M(\sigma)\neq M^{(n)}_0$.
\end{proof}

\begin{prop}\label{prop:web-Genocchi odd n,n=0}
  For $n>1$, we have $f(n,n)=0$. 
\end{prop}
\begin{proof}
  Let $\sigma$ be a web permutation of $[n]$ with $\sigma_1=n$.
  Since the elements \( 1 \) and \( n \) are contained in the same cycle, we have
  \( \sigma_n = 1 \) by Lemma~\ref{lem:cycle_minmax} and
  Theorem~\ref{thm:web=andre_cycle}. Then there is a marking at $(n,1)$
  in the grid configuration \( G(\sigma,\Cr(\sigma)) \).
  Observe that the vertical line and horizontal line starting from
  the cell $(n,1)$ do not make a crossing.
  Hence we deduce $\{1,2n\}\in M(\sigma)$,
  which implies that $M(\sigma)\neq M^{(n)}_0$.
\end{proof}

By Propositions~\ref{prop:web-Genocchi even k=0} and \ref{prop:web-Genocchi odd n,n=0}, we have
\[
  f(n)=\sum_{1\le k \le \lfloor n/2 \rfloor}f(n,2k-1).
\]
We now propose a conjecture that the values appearing in the Seidel triangle are $f(n,k)$.
\begin{conj}[Verified up to \( n=6 \)] \label{Conj: Genocchi}
  For \( n\ge 1 \), we have
  \[
    \begin{cases}
      f(2n-1,2k-1)= s_{2n-2, k},\\
      f(2n, 2k-1) = s_{2n-1, n-k+1}.
    \end{cases}
  \]
\end{conj}
This conjecture includes Nakamigawa's result. To elaborate, let \( \sigma \) be a
web permutation of \( [n] \) such that \( M(\sigma)=M^{(n)}_0 \) and
\( \sigma_1 = 1 \). Deleting the cycle \( (1) \) from \( \sigma \) and decreasing
each letter by 1, the resulting permutation is a web permutation of
\( [n-1] \) with \( M(\sigma)=M^{(n-1)}_0 \).
In addition, this correspondence is bijective, so we deduce \( f(n,1) = f(n-1) \).
Thus the conjecture implies \( f(n-1) = g_{n-1} \), which is Nakamigawa's result.

\section*{Acknowledgments}

The authors are grateful to Jang Soo Kim for several suggestions which
improved the manuscript.

\appendix
\section{Computations for some small \( n \)} \label{sec:appen}
In this appendix, we give several computational results for some small \( n \).

\subsection{The transition matrices}
All rows and columns are sorted with respect to the reverse lexicographic order
on their corresponding Dyck paths.
For example, let \( n = 3 \). The following is the list of 5 Dyck paths of
length \( 2n=6 \) sorted in the reverse lexicographic order:
\[
  \Dyck_{6} = \{ \NS\NS\NS\ES\ES\ES, \NS\NS\ES\NS\ES\ES, \NS\NS\ES\ES\NS\ES,
  \NS\ES\NS\NS\ES\ES, \NS\ES\NS\ES\NS\ES \}.
\]
Then rows and columns of the transition matrix for \( n=3 \) are indexed by
\( \NN_6 \) and \( \NC_6 \) in order as follows:
\[
  \NN_6 = \left\{ 
    \begin{tikzpicture}[scale=0.4]
      \Matching{1}{4}\Matching{2}{5}\Matching{3}{6}
      \foreach \i in {1,...,6}{
        \draw [fill] (\i,0) circle [radius=0.075] ;
      }
    \end{tikzpicture},~
    \begin{tikzpicture}[scale=0.4]
      \Matching{1}{3}\Matching{2}{5}\Matching{4}{6}
      \foreach \i in {1,...,6}{
        \draw [fill] (\i,0) circle [radius=0.075] ;
      }
    \end{tikzpicture},~
    \begin{tikzpicture}[scale=0.4]
      \Matching{1}{3}\Matching{2}{4}\Matching{5}{6}
      \foreach \i in {1,...,6}{
        \draw [fill] (\i,0) circle [radius=0.075] ;
      }
    \end{tikzpicture},~
    \begin{tikzpicture}[scale=0.4]
      \Matching{1}{2}\Matching{3}{5}\Matching{4}{6}
      \foreach \i in {1,...,6}{
        \draw [fill] (\i,0) circle [radius=0.075] ;
      }
    \end{tikzpicture},~
    \begin{tikzpicture}[scale=0.4]
      \Matching{1}{2}\Matching{3}{4}\Matching{5}{6}
      \foreach \i in {1,...,6}{
        \draw [fill] (\i,0) circle [radius=0.075] ;
      }
    \end{tikzpicture}
  \right\}
\]
and
\[
  \NC_6 = \left\{ 
    \begin{tikzpicture}[scale=0.4]
      \Matching{1}{6}\Matching{2}{5}\Matching{3}{4}
      \foreach \i in {1,...,6}{
        \draw [fill] (\i,0) circle [radius=0.075] ;
      }
    \end{tikzpicture},~
    \begin{tikzpicture}[scale=0.4]
      \Matching{1}{6}\Matching{2}{3}\Matching{4}{5}
      \foreach \i in {1,...,6}{
        \draw [fill] (\i,0) circle [radius=0.075] ;
      }
    \end{tikzpicture},~
    \begin{tikzpicture}[scale=0.4]
      \Matching{1}{4}\Matching{2}{3}\Matching{5}{6}
      \foreach \i in {1,...,6}{
        \draw [fill] (\i,0) circle [radius=0.075] ;
      }
    \end{tikzpicture},~
    \begin{tikzpicture}[scale=0.4]
      \Matching{1}{2}\Matching{3}{6}\Matching{4}{5}
      \foreach \i in {1,...,6}{
        \draw [fill] (\i,0) circle [radius=0.075] ;
      }
    \end{tikzpicture},~
    \begin{tikzpicture}[scale=0.4]
      \Matching{1}{2}\Matching{3}{4}\Matching{5}{6}
      \foreach \i in {1,...,6}{
        \draw [fill] (\i,0) circle [radius=0.075] ;
      }
    \end{tikzpicture}
  \right\}.
\]
We omit zeros in the strictly lower-triangular part of \( A \).
\begin{enumerate}[label=\roman*)]
  \item \( n=2 \)
  \[
    A = 
    \begin{bmatrix}
      1 & 1 \\
        & 1 
    \end{bmatrix}.
  \]
  \item \( n=3 \)
  \[
    A = 
    \begin{bmatrix}
      1 & 1 & 1 & 1 & 1 \\
        & 1 & 1 & 1 & 1 \\
        &   & 1 & 0 & 1 \\
        &   &   & 1 & 1 \\
        &   &   &   & 1
    \end{bmatrix}.
  \]
  \item \( n=4 \)
  \[
    A = 
    \begin{bmatrix}
      1 & 1 & 1 & 1 & 1 & 2 & 1 & 1 & 1 & 1 & 1 & 1 & 1 & 2  \\
        & 1 & 1 & 1 & 1 & 1 & 1 & 1 & 1 & 1 & 1 & 1 & 1 & 2  \\
        &   & 1 & 1 & 0 & 1 & 1 & 1 & 1 & 0 & 1 & 1 & 1 & 1  \\
        &   &   & 1 & 0 & 0 & 1 & 0 & 1 & 0 & 0 & 1 & 0 & 1  \\
        &   &   &   & 1 & 1 & 1 & 1 & 1 & 1 & 1 & 1 & 1 & 1  \\
        &   &   &   &   & 1 & 1 & 1 & 1 & 0 & 1 & 1 & 1 & 1  \\
        &   &   &   &   &   & 1 & 0 & 1 & 0 & 0 & 1 & 0 & 1  \\
        &   &   &   &   &   &   & 1 & 1 & 0 & 0 & 0 & 1 & 1  \\
        &   &   &   &   &   &   &   & 1 & 0 & 0 & 0 & 0 & 1  \\
        &   &   &   &   &   &   &   &   & 1 & 1 & 1 & 1 & 1  \\
        &   &   &   &   &   &   &   &   &   & 1 & 1 & 1 & 1  \\
        &   &   &   &   &   &   &   &   &   &   & 1 & 0 & 1  \\
        &   &   &   &   &   &   &   &   &   &   &   & 1 & 1  \\
        &   &   &   &   &   &   &   &   &   &   &   &   & 1 
    \end{bmatrix}.
  \]
\end{enumerate}
\subsection{Web permutations}
We present lists of all web permutations for \( n=2,3,4,5 \) with their
corresponding Dyck paths and noncrossing matchings.
Due to space limitation, the matchings are also represented as Dyck paths
via the bijection \( D:\NC_{2n}\rightarrow \Dyck_{2n} \).
\begin{enumerate}[label=\roman*)]
  \item \( n=2 \)
  \[
    \begin{tabular}{|c|c|c|}
      \hline
      Web permutations \( \sigma \) & \( D(\sigma) \) & \( M(\sigma) \) \\ \hline
      12 = (1)(2) & \( \NS\ES\NS\ES \) & \( \NS\ES\NS\ES \) \\ \hline
      21 = (1,2) & \( \NS\NS\ES\ES \) & \( \NS\NS\ES\ES \) \\ \hline
    \end{tabular}
  \]
  \item \( n=3 \)
  \[
    \begin{tabular}{|c|c|c|}
      \hline
      Web permutations \( \sigma \) & \( D(\sigma) \) & \( M(\sigma) \) \\ \hline
      123 = (1)(2)(3) & \( \NS\ES\NS\ES\NS\ES \) & \( \NS\ES\NS\ES\NS\ES \) \\ \hline
      132 = (1)(2,3) & \( \NS\ES\NS\NS\ES\ES \) & \( \NS\ES\NS\NS\ES\ES \) \\ \hline
      213 = (1,2)(3) & \( \NS\NS\ES\ES\NS\ES \) & \( \NS\NS\ES\ES\NS\ES \) \\ \hline
      231 = (1,2,3) & \( \NS\NS\NS\ES\ES\ES \) & \( \NS\NS\ES\NS\ES\ES \) \\ \hline
      321 = (1,3)(2) & \( \NS\NS\NS\ES\ES\ES \) & \( \NS\NS\NS\ES\ES\ES \) \\ \hline
    \end{tabular}
  \]
  \item \( n=4 \)
  \[
    \begin{tabular}{|c|c|c|}
      \hline
      Web permutations \( \sigma \) & \( D(\sigma) \) & \( M(\sigma) \) \\ \hline
      1234 = (1)(2)(3)(4) & \( \NS\ES\NS\ES\NS\ES\NS\ES \) & \( \NS\ES\NS\ES\NS\ES\NS\ES \) \\ \hline
      1243 = (1)(2)(3,4) & \( \NS\ES\NS\ES\NS\NS\ES\ES \) & \( \NS\ES\NS\ES\NS\NS\ES\ES \) \\ \hline
      1324 = (1)(2,3)(4) & \( \NS\ES\NS\NS\ES\ES\NS\ES \) & \( \NS\ES\NS\NS\ES\ES\NS\ES \) \\ \hline
      1342 = (1)(2,3,4) & \( \NS\ES\NS\NS\NS\ES\ES\ES \) & \( \NS\ES\NS\NS\ES\NS\ES\ES \) \\ \hline
      1432 = (1)(2,4)(3) & \( \NS\ES\NS\NS\NS\ES\ES\ES \) & \( \NS\ES\NS\NS\NS\ES\ES\ES \) \\ \hline
      2134 = (1,2)(3)(4) & \( \NS\NS\ES\ES\NS\ES\NS\ES \) & \( \NS\NS\ES\ES\NS\ES\NS\ES \) \\ \hline
      2143 = (1,2)(3,4) & \( \NS\NS\ES\ES\NS\NS\ES\ES \) & \( \NS\NS\ES\ES\NS\NS\ES\ES \) \\ \hline
      2314 = (1,2,3)(4) & \( \NS\NS\NS\ES\ES\ES\NS\ES \) & \( \NS\NS\ES\NS\ES\ES\NS\ES \) \\ \hline
      3214 = (1,3)(2)(4) & \( \NS\NS\NS\ES\ES\ES\NS\ES \) & \( \NS\NS\NS\ES\ES\ES\NS\ES \) \\ \hline
      3412 = (1,3)(2,4) & \( \NS\NS\NS\ES\NS\ES\ES\ES \) & \( \NS\ES\NS\ES\NS\ES\NS\ES \) \\ \hline
      2341 = (1,2,3,4) & \( \NS\NS\NS\NS\ES\ES\ES\ES \) & \( \NS\NS\ES\NS\ES\NS\ES\ES \) \\ \hline
      2431 = (1,2,4)(3) & \( \NS\NS\NS\NS\ES\ES\ES\ES \) & \( \NS\NS\ES\NS\NS\ES\ES\ES \) \\ \hline
      3241 = (1,3,4)(2) & \( \NS\NS\NS\NS\ES\ES\ES\ES \) & \( \NS\NS\NS\ES\ES\NS\ES\ES \) \\ \hline
      4231 = (1,4)(2)(3) & \( \NS\NS\NS\NS\ES\ES\ES\ES \) & \( \NS\NS\ES\NS\ES\NS\ES\ES \) \\ \hline
      3421 = (1,3,2,4) & \( \NS\NS\NS\NS\ES\ES\ES\ES \) & \( \NS\NS\NS\ES\NS\ES\ES\ES \) \\ \hline
      4321 = (1,4)(2,3) & \( \NS\NS\NS\NS\ES\ES\ES\ES \) & \( \NS\NS\NS\NS\ES\ES\ES\ES \) \\ \hline
    \end{tabular}
  \]
  \item \( n=5 \)
  \[
    \begin{tabular}{|c|c|c|}
      \hline
      Web permutations \( \sigma \) & \( D(\sigma) \) & \( M(\sigma) \) \\ \hline
      12345 = (1)(2)(3)(4)(5) & \( \NS\ES\NS\ES\NS\ES\NS\ES\NS\ES \) & \( \NS\ES\NS\ES\NS\ES\NS\ES\NS\ES \) \\ \hline
      12354 = (1)(2)(3)(4,5) & \( \NS\ES\NS\ES\NS\ES\NS\NS\ES\ES \) & \( \NS\ES\NS\ES\NS\ES\NS\NS\ES\ES \) \\ \hline
      12435 = (1)(2)(3,4)(5) & \( \NS\ES\NS\ES\NS\NS\ES\ES\NS\ES \) & \( \NS\ES\NS\ES\NS\NS\ES\ES\NS\ES \) \\ \hline
      12453 = (1)(2)(3,4,5) & \( \NS\ES\NS\ES\NS\NS\NS\ES\ES\ES \) & \( \NS\ES\NS\ES\NS\NS\ES\NS\ES\ES \) \\ \hline
      12543 = (1)(2)(3,5)(4) & \( \NS\ES\NS\ES\NS\NS\NS\ES\ES\ES \) & \( \NS\ES\NS\ES\NS\NS\NS\ES\ES\ES \) \\ \hline
      13245 = (1)(2,3)(4)(5) & \( \NS\ES\NS\NS\ES\ES\NS\ES\NS\ES \) & \( \NS\ES\NS\NS\ES\ES\NS\ES\NS\ES \) \\ \hline
      13254 = (1)(2,3)(4,5) & \( \NS\ES\NS\NS\ES\ES\NS\NS\ES\ES \) & \( \NS\ES\NS\NS\ES\ES\NS\NS\ES\ES \) \\ \hline
      13425 = (1)(2,3,4)(5) & \( \NS\ES\NS\NS\NS\ES\ES\ES\NS\ES \) & \( \NS\ES\NS\NS\ES\NS\ES\ES\NS\ES \) \\ \hline
      14325 = (1)(2,4)(3)(5) & \( \NS\ES\NS\NS\NS\ES\ES\ES\NS\ES \) & \( \NS\ES\NS\NS\NS\ES\ES\ES\NS\ES \) \\ \hline
      14523 = (1)(2,4)(3,5) & \( \NS\ES\NS\NS\NS\ES\NS\ES\ES\ES \) & \( \NS\ES\NS\ES\NS\ES\NS\ES\NS\ES \) \\ \hline
      13452 = (1)(2,3,4,5) & \( \NS\ES\NS\NS\NS\NS\ES\ES\ES\ES \) & \( \NS\ES\NS\NS\ES\NS\ES\NS\ES\ES \) \\ \hline
      13542 = (1)(2,3,5)(4) & \( \NS\ES\NS\NS\NS\NS\ES\ES\ES\ES \) & \( \NS\ES\NS\NS\ES\NS\NS\ES\ES\ES \) \\ \hline
      14352 = (1)(2,4,5)(3) & \( \NS\ES\NS\NS\NS\NS\ES\ES\ES\ES \) & \( \NS\ES\NS\NS\NS\ES\ES\NS\ES\ES \) \\ \hline
      15342 = (1)(2,5)(3)(4) & \( \NS\ES\NS\NS\NS\NS\ES\ES\ES\ES \) & \( \NS\ES\NS\NS\ES\NS\ES\NS\ES\ES \) \\ \hline
      14532 = (1)(2,4,3,5) & \( \NS\ES\NS\NS\NS\NS\ES\ES\ES\ES \) & \( \NS\ES\NS\NS\NS\ES\NS\ES\ES\ES \) \\ \hline
      15432 = (1)(2,5)(3,4) & \( \NS\ES\NS\NS\NS\NS\ES\ES\ES\ES \) & \( \NS\ES\NS\NS\NS\NS\ES\ES\ES\ES \) \\ \hline
      21345 = (1,2)(3)(4)(5) & \( \NS\NS\ES\ES\NS\ES\NS\ES\NS\ES \) & \( \NS\NS\ES\ES\NS\ES\NS\ES\NS\ES \) \\ \hline
      21354 = (1,2)(3)(4,5) & \( \NS\NS\ES\ES\NS\ES\NS\NS\ES\ES \) & \( \NS\NS\ES\ES\NS\ES\NS\NS\ES\ES \) \\ \hline
      21435 = (1,2)(3,4)(5) & \( \NS\NS\ES\ES\NS\NS\ES\ES\NS\ES \) & \( \NS\NS\ES\ES\NS\NS\ES\ES\NS\ES \) \\ \hline
      21453 = (1,2)(3,4,5) & \( \NS\NS\ES\ES\NS\NS\NS\ES\ES\ES \) & \( \NS\NS\ES\ES\NS\NS\ES\NS\ES\ES \) \\ \hline
      21543 = (1,2)(3,5)(4) & \( \NS\NS\ES\ES\NS\NS\NS\ES\ES\ES \) & \( \NS\NS\ES\ES\NS\NS\NS\ES\ES\ES \) \\ \hline
      23145 = (1,2,3)(4)(5) & \( \NS\NS\NS\ES\ES\ES\NS\ES\NS\ES \) & \( \NS\NS\ES\NS\ES\ES\NS\ES\NS\ES \) \\ \hline
      23154 = (1,2,3)(4,5) & \( \NS\NS\NS\ES\ES\ES\NS\NS\ES\ES \) & \( \NS\NS\ES\NS\ES\ES\NS\NS\ES\ES \) \\ \hline
      32145 = (1,3)(2)(4)(5) & \( \NS\NS\NS\ES\ES\ES\NS\ES\NS\ES \) & \( \NS\NS\NS\ES\ES\ES\NS\ES\NS\ES \) \\ \hline
      32154 = (1,3)(2)(4,5) & \( \NS\NS\NS\ES\ES\ES\NS\NS\ES\ES \) & \( \NS\NS\NS\ES\ES\ES\NS\NS\ES\ES \) \\ \hline
      34125 = (1,3)(2,4)(5) & \( \NS\NS\NS\ES\NS\ES\ES\ES\NS\ES \) & \( \NS\ES\NS\ES\NS\ES\NS\ES\NS\ES \) \\ \hline
      34152 = (1,3)(2,4,5) & \( \NS\NS\NS\ES\NS\NS\ES\ES\ES\ES \) & \( \NS\ES\NS\ES\NS\ES\NS\NS\ES\ES \) \\ \hline
      35142 = (1,3)(2,5)(4) & \( \NS\NS\NS\ES\NS\NS\ES\ES\ES\ES \) & \( \NS\ES\NS\ES\NS\NS\ES\NS\ES\ES \) \\ \hline
      23415 = (1,2,3,4)(5) & \( \NS\NS\NS\NS\ES\ES\ES\ES\NS\ES \) & \( \NS\NS\ES\NS\ES\NS\ES\ES\NS\ES \) \\ \hline
      24315 = (1,2,4)(3)(5) & \( \NS\NS\NS\NS\ES\ES\ES\ES\NS\ES \) & \( \NS\NS\ES\NS\NS\ES\ES\ES\NS\ES \) \\ \hline
      24513 = (1,2,4)(3,5) & \( \NS\NS\NS\NS\ES\ES\NS\ES\ES\ES \) & \( \NS\NS\ES\ES\NS\ES\NS\ES\NS\ES \) \\ \hline
      32415 = (1,3,4)(2)(5) & \( \NS\NS\NS\NS\ES\ES\ES\ES\NS\ES \) & \( \NS\NS\NS\ES\ES\NS\ES\ES\NS\ES \) \\ \hline
      42315 = (1,4)(2)(3)(5) & \( \NS\NS\NS\NS\ES\ES\ES\ES\NS\ES \) & \( \NS\NS\ES\NS\ES\NS\ES\ES\NS\ES \) \\ \hline
      42513 = (1,4)(2)(3,5) & \( \NS\NS\NS\NS\ES\ES\NS\ES\ES\ES \) & \( \NS\NS\ES\NS\ES\ES\NS\ES\NS\ES \) \\ \hline
      34215 = (1,3,2,4)(5) & \( \NS\NS\NS\NS\ES\ES\ES\ES\NS\ES \) & \( \NS\NS\NS\ES\NS\ES\ES\ES\NS\ES \) \\ \hline
      43215 = (1,4)(2,3)(5) & \( \NS\NS\NS\NS\ES\ES\ES\ES\NS\ES \) & \( \NS\NS\NS\NS\ES\ES\ES\ES\NS\ES \) \\ \hline
      35412 = (1,3,4)(2,5) & \( \NS\NS\NS\NS\ES\NS\ES\ES\ES\ES \) & \( \NS\ES\NS\ES\NS\NS\ES\ES\NS\ES \) \\ \hline
      43512 = (1,4)(2,3,5) & \( \NS\NS\NS\NS\ES\NS\ES\ES\ES\ES \) & \( \NS\ES\NS\NS\ES\ES\NS\ES\NS\ES \) \\ \hline
      45312 = (1,4)(2,5)(3) & \( \NS\NS\NS\NS\ES\NS\ES\ES\ES\ES \) & \( \NS\ES\NS\NS\ES\NS\ES\ES\NS\ES \) \\ \hline
    \end{tabular}
  \]
  \[\begin{tabular}{|c|c|c|}
    \hline
    Web permutations \( \sigma \) & \( D(\sigma) \) & \( M(\sigma) \) \\ \hline
      23451 = (1,2,3,4,5) & \( \NS\NS\NS\NS\NS\ES\ES\ES\ES\ES \) & \( \NS\NS\ES\NS\ES\NS\ES\NS\ES\ES \) \\ \hline
      23541 = (1,2,3,5)(4) & \( \NS\NS\NS\NS\NS\ES\ES\ES\ES\ES \) & \( \NS\NS\ES\NS\ES\NS\NS\ES\ES\ES \) \\ \hline
      24351 = (1,2,4,5)(3) & \( \NS\NS\NS\NS\NS\ES\ES\ES\ES\ES \) & \( \NS\NS\ES\NS\NS\ES\ES\NS\ES\ES \) \\ \hline
      25341 = (1,2,5)(3)(4) & \( \NS\NS\NS\NS\NS\ES\ES\ES\ES\ES \) & \( \NS\NS\ES\NS\ES\NS\ES\NS\ES\ES \) \\ \hline
      24531 = (1,2,4,3,5) & \( \NS\NS\NS\NS\NS\ES\ES\ES\ES\ES \) & \( \NS\NS\ES\NS\NS\ES\NS\ES\ES\ES \) \\ \hline
      25431 = (1,2,5)(3,4) & \( \NS\NS\NS\NS\NS\ES\ES\ES\ES\ES \) & \( \NS\NS\ES\NS\NS\NS\ES\ES\ES\ES \) \\ \hline
      32451 = (1,3,4,5)(2) & \( \NS\NS\NS\NS\NS\ES\ES\ES\ES\ES \) & \( \NS\NS\NS\ES\ES\NS\ES\NS\ES\ES \) \\ \hline
      32541 = (1,3,5)(2)(4) & \( \NS\NS\NS\NS\NS\ES\ES\ES\ES\ES \) & \( \NS\NS\NS\ES\ES\NS\NS\ES\ES\ES \) \\ \hline
      42351 = (1,4,5)(2)(3) & \( \NS\NS\NS\NS\NS\ES\ES\ES\ES\ES \) & \( \NS\NS\ES\NS\ES\NS\ES\NS\ES\ES \) \\ \hline
      52341 = (1,5)(2)(3)(4) & \( \NS\NS\NS\NS\NS\ES\ES\ES\ES\ES \) & \( \NS\NS\ES\NS\NS\ES\ES\NS\ES\ES \) \\ \hline
      42531 = (1,4,3,5)(2) & \( \NS\NS\NS\NS\NS\ES\ES\ES\ES\ES \) & \( \NS\NS\ES\NS\ES\NS\NS\ES\ES\ES \) \\ \hline
      52431 = (1,5)(2)(3,4) & \( \NS\NS\NS\NS\NS\ES\ES\ES\ES\ES \) & \( \NS\NS\ES\NS\NS\ES\NS\ES\ES\ES \) \\ \hline
      34251 = (1,3,2,4,5) & \( \NS\NS\NS\NS\NS\ES\ES\ES\ES\ES \) & \( \NS\NS\NS\ES\NS\ES\ES\NS\ES\ES \) \\ \hline
      35241 = (1,3,2,5)(4) & \( \NS\NS\NS\NS\NS\ES\ES\ES\ES\ES \) & \( \NS\NS\NS\ES\ES\NS\ES\NS\ES\ES \) \\ \hline
      43251 = (1,4,5)(2,3) & \( \NS\NS\NS\NS\NS\ES\ES\ES\ES\ES \) & \( \NS\NS\NS\NS\ES\ES\ES\NS\ES\ES \) \\ \hline
      53241 = (1,5)(2,3)(4) & \( \NS\NS\NS\NS\NS\ES\ES\ES\ES\ES \) & \( \NS\NS\NS\ES\NS\ES\ES\NS\ES\ES \) \\ \hline
      34521 = (1,3,5)(2,4) & \( \NS\NS\NS\NS\NS\ES\ES\ES\ES\ES \) & \( \NS\NS\NS\ES\NS\ES\NS\ES\ES\ES \) \\ \hline
      35421 = (1,3,4,2,5) & \( \NS\NS\NS\NS\NS\ES\ES\ES\ES\ES \) & \( \NS\NS\NS\ES\NS\NS\ES\ES\ES\ES \) \\ \hline
      43521 = (1,4,2,3,5) & \( \NS\NS\NS\NS\NS\ES\ES\ES\ES\ES \) & \( \NS\NS\NS\NS\ES\ES\NS\ES\ES\ES \) \\ \hline
      53421 = (1,5)(2,3,4) & \( \NS\NS\NS\NS\NS\ES\ES\ES\ES\ES \) & \( \NS\NS\NS\ES\NS\ES\NS\ES\ES\ES \) \\ \hline
      45321 = (1,4,2,5)(3) & \( \NS\NS\NS\NS\NS\ES\ES\ES\ES\ES \) & \( \NS\NS\NS\NS\ES\NS\ES\ES\ES\ES \) \\ \hline
      54321 = (1,5)(2,4)(3) & \( \NS\NS\NS\NS\NS\ES\ES\ES\ES\ES \) & \( \NS\NS\NS\NS\NS\ES\ES\ES\ES\ES \) \\ \hline
    \end{tabular}
  \]
\end{enumerate}

\bibliographystyle{alpha}

\end{document}